\documentclass[10pt, oneside]{amsart}   	
\usepackage{amsmath,amssymb,verbatim,tikz}
\usepackage{geometry}

\usepackage[numbers]{natbib}
\usepackage[mathscr]{euscript}
\usepackage{mathtools}
\usepackage{ytableau}
\usepackage{xcolor,shuffle}
\usepackage{multirow}

\usepackage{hyperref, enumitem}
\usepackage{xcolor,shuffle}
\usepackage{longtable}
\usepackage{colortbl}
\usepackage[capitalize]{cleveref}
\usepackage{tikz-cd}
\usepackage{ytableau}

\usepackage[utf8]{inputenc}

\newtheorem{thm}{Theorem}[section]
\newtheorem{lemma}[thm]{Lemma}
\newtheorem{prop}[thm]{Proposition}
\newtheorem{cor}[thm]{Corollary}

\newtheorem{question}[thm]{Question}
\newtheorem{definition}[thm]{Definition}

\theoremstyle{definition}
\newtheorem{example}[thm]{Example}

\theoremstyle{remark}
\newtheorem{remark}[thm]{Remark}

\def\R{{\mathbb{R}}}
\def\Z{{\mathbb{Z}}}
\def\F{{\mathbb{F}}}

\def\aa{\mathbf{a}}
\def\bb{\mathbf{b}}
\def\cc{\mathbf{c}}
\def\kk{\mathbf{k}}
\def\vv{\mathbf{v}}

\def\triv{\mathbf{1}}

\def\AAA{\mathbf{A}}
\def\BBB{\mathbf{B}}
\def\CCC{\mathbf{C}}

\def\sgn{\mathrm{sgn}}

\def\Des{\mathrm{Des}}

\def\span{\mathrm{span}}

\def\virtchars{\mathcal{C}}
\def\flags{\mathcal{F}}

\def\specht{\chi}
\def\freelrb{{\mathscr{F}}}
\def\qfreelrb{\mathscr{F}^{(q)}}
\def\derangement{\mathfrak{d}}
\def\derangementrep{{\mathscr{D}}}

\def\ch{{\mathrm{ch}}}

\def\symm{\mathfrak{S}}
\def\higherLie{\mathfrak{L}}

\newcommand{\qbinom}[3]{\left[ \begin{matrix} #1 \\ #2 \end{matrix}\right]_{#3}}

\title[Invariant Theory for the Free left-regular band and a $q$-analogue]{Invariant Theory for the Free left-regular band \\ and a $q$-analogue}

\author{Sarah Brauner}
\author{Patricia Commins}
\author{Victor Reiner}
\address{School of Mathematics, University of Minnesota, Minneapolis, MN 55455, USA}
\email{braun622@umn.edu, commi010@umn.edu, reiner@umn.edu}

\begin{document}

\begin{abstract}
We examine from an invariant theory viewpoint the monoid algebras for two monoids having large symmetry groups.  The first monoid is the 
{\it free left-regular band} on $n$ letters,
defined on the set of all {\it injective} words, that is, the words with at 
most one occurrence of each letter. This monoid carries the action of the symmetric group.
The second monoid is one of its $q$-analogues, considered by K. Brown, carrying an action of the finite general linear group. 
In both cases, we show that the invariant subalgebras are
 semisimple commutative algebras, and characterize them using \emph{Stirling} and \emph{q-Stirling numbers}.  
 
 We then use results from the theory of random walks and random-to-top shuffling to decompose the entire monoid algebra into irreducibles, simultaneously as a module over the invariant ring and as a group representation. 
 Our irreducible decompositions are described in terms of \emph{derangement symmetric functions} introduced by D\'esarm\'enien and Wachs. 
\end{abstract}



\dedicatory{To the memory of Georgia Benkart.}

\maketitle

\section{Introduction}
\label{intro-section}

Motivated by results on mixing times for shuffling algorithms on permutations, Bidigare  \cite{Bidigare-thesis} and Bidigare, Hanlon and Rockmore \cite{BHR} developed a complete spectral analysis for a class of random walks on chambers of a hyperplane arrangement.  Their work relied heavily on the {\it Tits semigroup} structure on the cones of the arrangement.
Later Brown \cite{BrownOnLRBs} generalized their analysis
to random walks coming from
semigroups $\freelrb$ which form a {\it left-regular band (LRB)}, meaning that $x^2=x$ for all $x$ and $xyx=xy$ for all $x,y$ in $\freelrb$.

Here we study two examples
of left-regular bands $M$, related to those discussed by Brown, having actions of large groups of monoid automorphisms $G$: 
\begin{itemize}
    \item the {\it free LRB on $n$ letters} \cite[\S1.3]{BrownOnLRBs}, denoted $\freelrb_n$,
    with $G$ the symmetric group $\symm_n$, and
    \item a $q$-analogue $\qfreelrb_n$ related to monoids in \cite{BrownOnLRBs}, and $G$ the general linear group $GL_n:=GL_n(\F_q)$.
\end{itemize}
For both monoids $M=\freelrb_n,\qfreelrb_n$, we examine the {\it monoid algebra} $R:=\kk M$ with coefficients in a commutative ring $\kk$, and answer the
two {\it main questions of invariant theory} for $G$ acting on $R$:

\begin{question}
\label{first-main-invariant-theory-question}
What is the structure of the invariant subalgebra $R^G$?
\end{question}

\begin{question}
\label{second-main-invariant-theory-question}
What is the structure of $R$, simultaneously as an $R^G$-module
and $G$-representation?
\end{question}

Section~\ref{invariant-subalgebra-section} answers
Question~\ref{first-main-invariant-theory-question}
with our first main result,
using the combinatorics of {\it Stirling} and {\it $q$-Stirling numbers}.  We paraphrase it here; see Theorem~\ref{semisimplicity-theorem} for a more precise
statement.  

\begin{thm}
\label{paraphrased-semisimplicity-thm}
Consider either monoid $M=\freelrb_n,\qfreelrb_n$ with
symmetry groups $G=\symm_n, GL_n$,
and assume that $\kk$ is a field in which $|G|$
is invertible.

\begin{itemize}
    \item The invariant subalgebra 
$R^G$ is a commutative subalgebra of $R$ generated by a single element; call this element $x$ for $M = \freelrb_n$ and $x^{(q)}$ for $M = \qfreelrb_n$.

\item The elements $x, x^{(q)}$ have
minimal polynomials
$$
f(X)=
\begin{cases}
X(X-1)(X-2) \cdots (X-n)
& \text{ if }M=\freelrb_n,\\
X(X-[1]_q)(X-[2]_q) \cdots (X-[n]_q)
& \text{ if }M=\qfreelrb_n.\\
\end{cases}
$$
where 
$
[m]_q:=1 + q + \cdots + q^{m-1}
$
is a standard {\it $q$-analogue} 
of the integer $m \geq 0$.

\item
In particular, $R^G \cong \kk[X]/(f(X))$,
and $R^G$ acts semisimply on $R$, with
\begin{itemize}
    \item 
$x$-eigenvalues $0,1,2,\ldots,n$ on $R=\kk\freelrb_n$,
\item $x^{(q)}$-eigenvalues $[0]_q,[1]_q,\ldots,[n]_q$  on
$R=\kk\qfreelrb_n$.
\end{itemize}
\end{itemize}
\end{thm}

Since the above hypothesis that $|G|$ is invertible in $\kk$ also implies that
$\kk G$ acts semisimply by Maschke's Theorem, this leads to our next goal:  a complete answer to Question~\ref{second-main-invariant-theory-question} above, decomposing the monoid algebra $R$ into simple modules for the simultaneous (commuting) actions of $R^G$ and $G$. The fact that $R^G$ is generated by a single, semisimple element $x$ (respectively $x^{(q)}$) reduces this problem to understanding each eigenspace of $x$ (respectively $x^{(q)}$) as a $\kk G$-module. 

To describe these $\kk G$-modules, recall that irreducible representations $\{ \chi^{\lambda} \} $ of $\symm_n$ are indexed by partitions $\lambda$ of $n$ and let 
$
\virtchars(\symm) := 
\bigoplus_{n=0}^\infty
\virtchars(\symm_n),
$
where $\virtchars(\symm_n)$ denotes
the $\Z$-module of {\it virtual characters}
of $\symm_n$. Then the classical \emph{Frobenius characteristic map} $\ch$ is an algebra isomorphism between $\virtchars(\symm)$ and the ring of symmetric functions $\Lambda$.  It has $\ch(\chi^\lambda)=s_\lambda$, the {\it Schur function}, and the trivial representation $\triv_n$ has $\ch(\triv_n)=h_n$, the {\it complete homogeneous} symmetric function.

There is a parallel and $q$-analogous story for a subset of irreducible representations $\{ \chi^{\lambda}_q \}$ of $GL_n$ called the \emph{unipotent representations}, also indexed by partitions $\lambda$ of $n$. These are the irreducible constitutents of the $GL_n$-permutation action on
the set $GL_n/B=\flags(V)$ of {\it complete flags of subspaces} in 
$V=(\F_q)^n$. Here too there is a $q$-\emph{Frobenius characteristic map} $\ch_q$ that defines an algebra isomorphism between $\virtchars(GL):=\bigoplus_{n=0}^\infty \virtchars(GL_n)$ and $\Lambda$, where  $\virtchars(GL_n)$ is the free $\Z$-submodule of the
class functions on $GL_n$ spanned by the unipotent characters
$\{\specht_q^\lambda\}$. As one might hope, $\ch_q(\specht_q^\lambda) = s_\lambda$ and $\ch_q(\triv_{GL_n}) = h_n$, where $\triv_{GL_n}$ is the trivial representation of $GL_n$. 

This allows us to phrase parallel answers
 to Question~\ref{second-main-invariant-theory-question}, in terms of an important family of symmetric functions introduced by D\'esarm\'enien and Wachs in \cite{FrenchDesarmenienWachs} which we will call the \emph{D\'esarm\'enien-Wachs derangement symmetric functions} $\{ \derangement_n \}_{n = 0, 1,2, \ldots}$, reviewed in Section~\ref{sec:derangement}.   Here $\derangement_n$ is both the Frobenius image of an
 $\symm_n$-representation $\derangementrep_n$ that we call the \emph{Derangement representation}, as well as the $q$-Frobenius image of a $q$-analogous $GL_n$-representation $\derangementrep_n^{(q)}$. As the name suggests, these representations have dimensions counted by the \emph{derangement numbers} and \emph{$q$-derangement numbers}, respectively\footnote{There are two natural families of $\kk\symm_n$-modules whose dimensions are the derangement numbers, discussed in \cite[Thm. 1.2]{HershReiner}.  The representation $\derangementrep_n$ here is the one with character $\widehat{\mathrm{Lie}}_n$ in the notation of \cite[eqn. (1)]{HershReiner}.}. They have irreducible decomposition
\[ \derangementrep_n \cong \bigoplus_Q \specht^{\lambda(Q)}, \hspace{3em}\derangementrep^{(q)}_n \cong \bigoplus_Q \specht_q^{\lambda(Q)},  \]
where $Q$ runs through all standard Young tableaux of size $n$ whose {\it first ascent} is even \cite{ReinerWebb}. Derangement symmetric functions have connections to many well-studied objects in combinatorics such as the complex of injective words \cite{ReinerWebb}, random-to-top and random-to-random shuffling \cite{Uyemura-Reyes}, higher Lie characters \cite{Uyemura-Reyes}, and configuration spaces \cite{HershReiner}; see Section~\ref{sec:derangement}. We add to this list by showing they form crucial building blocks for the invariant theory of $\kk \freelrb_n$ and $\kk \qfreelrb_n$.

Section~\ref{main-results-section} derives the following answer to Question~\ref{second-main-invariant-theory-question}, paraphrased here---see Theorem~\ref{precise-version-of-semigroupeigen} for a more precise statement.

\begin{thm}
\label{thm:semigroupeigen}
Let $\kk$ be a field whose characteristic does not divide $|G|$.  Then
when $x, x^{(q)}$ act on $\kk \freelrb_{n}, \kk \qfreelrb_{n}$,
for each $j=0,1,2,\ldots,n$, the
$j$-eigenspace for $x$ and
$[j]_q$-eigenspace for $x^{(q)}$ carry $G$-representations with
the same Frobenius map images
$$
\ch \ker\left( (x-j)|_{\kk \freelrb_n} \right)
= \sum_{\ell=j}^n \ h_{n-\ell} \cdot h_{j} \cdot \derangement_{\ell-j}
= \ch_q \ker\left( (x^{(q)}-[j]_q)|_{\kk \qfreelrb_n} \right).
$$
\end{thm}

Our proofs employ techniques that go back to discussion between Michelle Wachs and the third author in the analysis of random-to-top shuffling, and have been employed more recently by  Dieker and Saliola \cite{dieker2018spectral} and
Lafreni\'ere \cite{Lafreniere} in the analysis of
random-to-random shuffling and a generalization.  The method constructs eigenvectors of $x, x^{(q)}$ acting on $\freelrb_n, \qfreelrb_n$ from nullvectors associated to the analogous operators for smaller values of $n$. Combining these ideas with various filtrations on $\kk M$ allows us to describe the eigenspaces as parabolic inductions of derangement representations in a conceptual way, avoiding character computations. 

The remainder of the paper proceeds as follows.
\begin{itemize}
    \item Section~\ref{invariant-subalgebra-section} introduces the monoid algebras of interest, $R = \kk \freelrb_n, \kk \qfreelrb_n$ and proves Theorem~\ref{paraphrased-semisimplicity-thm} describing in parallel the invariant subalgebras $R^G$ for $G = \symm_n, GL_n$.
    \item Section~\ref{representations-review-section} reviews the relation between symmetric functions, representations of $\symm_n$ and unipotent representations of $GL_n$. It also introduces the derangement symmetric functions $\derangement_n$, and describes some of their many definitions and guises. 
\item
Section~\ref{main-results-section} proves Theorem~\ref{thm:semigroupeigen}, simultaneously decomposing the monoid algebra $R$ into simple modules for $R^G$ and $\kk G$, with arguments in parallel for $R=\kk\freelrb_n$ and $R=\kk\qfreelrb_n$.
\end{itemize}

\section*{Acknowledgements}
The authors are very grateful to Darij Grinberg for helpful references and conversations, to Franco Saliola for useful discussions, and to Peter Webb for organizing a reading seminar on Benjamin Steinberg's text \cite{Steinberg} that helped spark this project.  The third author thanks Michelle Wachs for helpful discussions on random-to-top shuffling. First and second authors are supported by NSF Graduate Research Fellowships,
and third author by NSF grant DMS-2053288.
\section{Definitions, background and the answer to Question~\ref{first-main-invariant-theory-question}}
\label{invariant-subalgebra-section}

We introduce the monoids $M=\freelrb_n, \qfreelrb_n$,
the symmetries $G=\symm_n, GL_n$ of the monoid algebras $R=\kk M$, and analyze the invariant rings $R^G$. Useful references are Brown \cite{BrownOnLRBs},  B. Steinberg \cite{Steinberg}.

\subsection{The monoids $\freelrb_n$ and  $\qfreelrb_n$}

\begin{definition} 
\label{defn:freeLRB} \rm
The {\it free left-regular band (or LRB)} on $n$ letters $\freelrb_n$ (see 
\cite[\S1.3]{BrownOnLRBs}, \cite[\S14.3.1]{Steinberg}) consists, as a set, of all words
$\aa=(a_1,a_2,\ldots,a_\ell)$ with letters $a_i$
from $\{1,2,\ldots,n\}$ and no repeated letters,
that is, $a_i \neq a_j$ for $1 \leq i<j \leq n$.
Here the {\it length} $\ell(\aa):=\ell$ lies
anywhere in the range $0 \leq \ell \leq n$.
The set $\freelrb_n$ becomes a semigroup under the following operation:
if $\bb=(b_1,\ldots,b_m)$ is another word in
$\freelrb_n$, then their product is
$$
\aa \cdot \bb:=(a_1,\ldots,a_\ell,b_1,\ldots,b_m)^{\wedge},
$$
where we have borrowed the notation from Brown \cite{BrownOnLRBs} that for a sequence $\cc=(c_1,\ldots,c_p)$, the subsequence
$\cc^\wedge=(c_1,\ldots,c_p)^\wedge$
is obtained by removing any letter 
$c_i$ that appears already in the prefix $(c_1,c_2,\ldots,c_{i-1})$.
One can check that the empty word $()$ is an identity element for
this operation, and hence $\freelrb_n$ is not only a semigroup, but a {\it monoid}.
\end{definition}

\begin{definition} 
\label{defn:q-freeLRB}
\rm
The $q$-analogue of $\freelrb_n$ that we will consider will be denoted
$\qfreelrb_n$.  As a set, it consists
of all partial flags of subspaces
$\AAA=(A_1,A_2,\ldots,A_\ell)$ where $A_i$ is an $i$-dimensional $\F_q$-linear subspace of $(\F_q)^n$, and
$A_1 \subset A_2 \subset \cdots \subset A_\ell$.
Again the length $\ell(\AAA):=\ell$ lies in the range $0 \leq \ell \leq n$.  The set $\qfreelrb_n$
becomes a semigroup under the following operation:
if $\BBB=(B_1,\ldots,B_m)$ is another such flag
in $\qfreelrb_n$, then
$$
\AAA \cdot \BBB:=
(A_1,\ldots,A_\ell,A_\ell+B_1,A_\ell+B_2,\ldots,A_\ell+B_m)^\wedge
$$
using a similar notation as before:  for a sequence $\CCC=(C_1,\ldots,C_p)$ of nested
subspaces $C_1 \subseteq C_2 \subseteq \cdots \subseteq C_p$, the subsequence $\CCC^\wedge$
is obtained by
by removing any subspace $C_i$ that appears already
in the prefix $(C_1,C_2,\ldots,C_{i-1})$.
As above, $\qfreelrb_n$ is not only a semigroup but a {\it monoid} since the empty flag $()$ is an identity element.
\end{definition}

\begin{remark}
\label{three-q-analogues-warning-remark}
{\bf Warning:} 
Brown \cite[\S1.4, \S5]{BrownOnLRBs}
introduced two other monoids
$\freelrb_{n,q}$ and $\overline{\freelrb}_{n,q}$,
closely related to $\qfreelrb_n$.  All three are different
$q$-analogues of $\freelrb_n$, related as follows.

As a set, Brown's first $q$-analogue $\freelrb_{n,q}$ consists of all sequences $\vv=(v_1,v_2,\ldots,v_\ell)$ of linearly independent vectors in $(\F_q)^n$.  For another
sequnce $\vv'=(v'_1,v'_2,\ldots,v'_m)$, one defines their product  
$$
\vv \cdot \vv':=(v_1,v_2,\ldots,v_\ell,v_1',v_2'\ldots,v_m')^\wedge
$$
where $(u_1,\ldots,u_p)^\wedge$ is obtained
by removing any $u_i$ which is dependent upon the preceding vectors $(u_1,\ldots,u_{i-1})$.
One may regard the monoid $\qfreelrb_n$
as a quotient monoid of $\freelrb_{n,q}$
via the surjection
$$
\begin{array}{rcl}
\freelrb_{n,q} & \twoheadrightarrow &
\qfreelrb_n\\
(v_1,v_2,\ldots,v_\ell) & \longmapsto & (A_1,A_2,\ldots,A_\ell)  \text{ where }A_i:=\F_q v_1 + \F_q v_2 + \cdots + \F_q v_i.
\end{array}
$$

Brown's second $q$-analogue $\overline{\freelrb}_{n,q}$
turns out to be a further quotient of
either $\freelrb_{n,q}$ or $\qfreelrb_n$, whose motivation he explains in \cite[\S 5.1,5.2]{BrownOnLRBs}.  It is $q$-analogous to a certain quotient monoid of $\freelrb_n$ that he
denotes $\overline{\freelrb}_n$, which one could define as follows:  
the monoid quotient map $\freelrb_n \twoheadrightarrow \overline{\freelrb}_n$ identifies the longest
words, those of length $n$, with their prefix word of length $n-1$:
$$
\overline{(a_1,a_2,\ldots,a_{n-1},a_n)}=
\overline{(a_1,a_2,\ldots,a_{n-1})}.
$$
One can then define Brown's second $q$-analogue
$\overline{\freelrb}_{n,q}$ as a quotient 
of $\qfreelrb_n$, where
the monoid quotient map $\qfreelrb_n \twoheadrightarrow \overline{\freelrb}_{n,q}$ identifies a
complete flag of length $n$ with the flag of length $n-1$ that omits the (improper) subspace $(\F_q)^n$ at the end:
$$
\overline{(A_1, A_2,\cdots,A_{n-1},(\F_q)^n)}=
\overline{(A_1,A_2,\ldots,A_{n-1})}.
$$
\end{remark}

\subsection{Symmetries of the monoid algebras}

Let $\kk$ be a commutative ring with $1$.  For any finite monoid $M$ (such as 
$M=\freelrb_n, \qfreelrb_n$), 
the {\it monoid algebra} $R=\kk M$ is
the free $\kk$-module with basis elements given by
the elements $\aa$ of $M$, and multiplication extended
$\kk$-linearly from the monoid operation on the basis elements:
$$
\left(
\sum_\aa p_\aa  \,\, \aa
\right)
\left(
\sum_\bb  q_\bb\,\, \bb
\right)
=
\sum_{\aa,\bb}
p_\aa q_\bb  \,\, \aa \cdot \bb 
=
\sum_\cc \left( \sum_{\aa \cdot \bb=\cc} p_\aa q_\bb  \right) \cc.
$$
Note that any group $G$ of monoid automorphisms of $M$ acts as ring automorphisms on $R=\kk M$.  In particular, the symmetric group $\symm_n$ permuting letters $\{1,2,\ldots,n\}$ acts on $\freelrb_n$ via
$$
w(a_1,\ldots,a_\ell)=(w(a_1),\ldots,w(a_\ell)).
$$
Similarly, the finite general linear group $GL_n:=GL_n(\F_q)$  acts
on $\qfreelrb_n$ by 
$$
g(A_1,\ldots,A_\ell)=(g(A_1),\ldots,g(A_\ell)).
$$
Our first goal is to analyze the $G$-invariant subalgebras $R^G$ in both cases.

\subsection{The invariant subalgebras $R^G$
and Question~\ref{first-main-invariant-theory-question}}

Since the groups $G$ permute the monoid elements $M$, the monoid
algebra $R=\kk M$ becomes a permutation representation of $G$.
Therefore the invariant subalgebra $R^G$ has as a $\kk$-basis
the orbit sums $\{ \sum_{\aa \in \mathcal{O}} \aa\}$
as one runs through all $G$-orbits $\mathcal{O}$ on $M$.
For both monoids $M=\freelrb_n,\qfreelrb_n$, one can easily identify the $G$-orbits, since the groups $G=\symm_n,GL_n$ act transitively on the subsets 
$$
\begin{aligned}
\freelrb_{n,\ell}&:=\{\aa \in \freelrb_n: \ell(\aa)=\ell\},\\
\qfreelrb_{n,\ell}&:=\{\AAA \in \qfreelrb_n: \ell(\AAA)=\ell\}.
\end{aligned}
$$
Thus the $G$-invariant
subalgebras $R^G$ have  
$\kk$-bases 
$\{x_\ell\}_{\ell=0,1,\ldots,n},$ and $
\{x^{(q)}_\ell\}_{\ell=0,1,\ldots,n}$,
 defined by
\begin{equation}
\label{orbit-sum-bases}
\begin{aligned}
x_\ell&:=\sum_{\aa \in \freelrb_{n,\ell}} \aa,\\
x^{(q)}_\ell&:=\sum_{\AAA \in \qfreelrb_{n,\ell}} \AAA. 
\end{aligned}
\end{equation}
\begin{example}
Let $q = 2$, $n=3$, $\ell = 1$,
and let $e_1,e_2,e_3$ be standard basis vectors for $V=(\F_2)^3$. Using the notation $\langle v_1,v_2,\ldots,v_m\rangle$ for the $\F_q$-span
of the vectors $\{v_1,v_2,\ldots,v_m\}$ in $V$,
one has 
\[
\begin{aligned} x^{(2)}_{1} =& (\langle e_1 \rangle ) + (\langle e_2 \rangle ) + (\langle e_3 \rangle ) + (\langle e_1 + e_2 \rangle )+ (\langle e_1 + e_3 \rangle ) + (\langle e_2 + e_3 \rangle ) + (\langle e_1 + e_2 + e_3 \rangle ).
\end{aligned} \]
\end{example}

It will be convenient to adopt the convention that $x_{n+1}:=0=:x^{(q)}_{n+1}$.

Using the $\kk$-bases in \eqref{orbit-sum-bases} for $(\kk\freelrb_n)^{\symm_n}$ and $(\kk \qfreelrb_n)^{GL_n}$, there is a simple
rule for multiplication by the elements
$$
\begin{array}{rccclcl}
x&:=&x_1&=&\sum_{i=1}^n (i) &= (1)+(2)+\cdots+(n),\\
x^{(q)}&:=&x^{(q)}_1&=&\sum_{\text{lines }L \subset (\F_q)^n} (L).&
\end{array}
$$
To state the rule, recall a standard $q$-analogue of nonnegative integers
$$
[n]_q:=1+q+q^2+\cdots+q^{n-1}.
$$

\begin{lemma}
\label{x-multiplication}
Inside $R^G$ for the monoid algebras
$R=\kk M$ with $M=\freelrb_n, \qfreelrb_n$, the
elements $x$ and $x^{(q)}$ act on the (ordered) $\kk$-bases 
\eqref{orbit-sum-bases}
as follows: for $\ell=0,1,\ldots,n$,
$$
\begin{aligned}
x \cdot x_\ell&= \ell x_{\ell} + x_{\ell+1},\\
x^{(q)} \cdot x^{(q)}_\ell&= [\ell]_q \, x_{\ell}^{(q)} + q^\ell \, x_{\ell+1}^{(q)}.
\end{aligned}
$$
In other words, $x$ and $x^{(q)}$ act on $R^G$, in the ordered bases above, via the matrices:
$$
x=\left[
\begin{matrix}
0&  &  &  &       &   &\\
1& 1&  &  &       &   &\\
 & 1& 2&  &       &   &\\
 &  & 1& \ddots     &   &\\
 &  &  &  &n-1&\\
 &  &  &       &  1&n
\end{matrix}
\right]
\qquad 
x^{(q)}=\left[
\begin{matrix}
[0]_q&  &  &  &       &   &\\
q^0& [1]_q&  &  &       &   &\\
 & q^1&  [2]_q&  &       &   &\\
 &  & q^2& \ddots     &   &\\
 &  &  &  &[n - 1]_q &\\
 &  &  &       &  q^{n - 1}& [n]_q
\end{matrix}
\right].
$$
\end{lemma}
\begin{proof}
Note that the product $x \cdot x_\ell$ is $G$-invariant,
and is a sum of terms $\aa$ of length $\ell$ or $\ell+1$,
so it must have the form $c \cdot x_\ell+ d \cdot x_{\ell+1}$
for some constants $c,d$ in $\kk$. 
The constant $d=1$ since any word $\aa=(a_1,a_2,\ldots,a_{\ell+1})$
of length $\ell+1$
arises uniquely as $(a_1) \cdot (a_2,\ldots,a_{\ell+1})$.
The constant $c=\ell$ since 
any word $(a_1,a_2,\ldots,a_{\ell})$
of length $\ell$ arises in $\ell$ ways,
from these products:
$$
\begin{aligned}
(a_1) & \cdot (\underline{a_1},a_2,a_3,a_4,\ldots,a_{\ell}),\\
(a_1) & \cdot (a_2,\underline{a_1},a_3,a_4,\ldots,a_{\ell}),\\
(a_1) & \cdot (a_2,a_3,\underline{a_1},a_4,\ldots,a_{\ell}),\\
&\vdots
\\
(a_1) & \cdot (a_2,a_3,a_4,\ldots,a_{\ell},\underline{a_1})
\end{aligned}. 
$$

For the $q$-analogous formula, one argues similarly that
$$
x^{(q)} \cdot x^{(q)}_\ell=c \cdot x^{(q)}_\ell+ d \cdot x^{(q)}_{\ell+1}
$$
for some constants $c,d$ in $\kk$.  We first show that the constant $d=q^\ell$. Any flag $\AAA=(A_1,A_2,\ldots,A_{\ell+1})$
of length $\ell+1$
arises from products of the form
$(A_1) \cdot (B_1,B_2\ldots,B_{\ell})$
where the flag $B_1 \subset B_2 \subset \cdots
\subset B_{\ell}$ satisfies $A_1+B_i=A_{i+1}$
for $i=1,2,\ldots,\ell$.  If one picks $B_1,B_2,\cdots,B_{\ell}$
sequentially, then having chosen $B_{i-1}$, one must choose $B_i$ so that $B_i/B_{i-1}$ is any line
inside the $2$-dimensional quotient space $A_{i+1}/B_{i-1}$ other than the line $(A_1+B_{i-1})/B_{i-1}$.  Since there are $q+1$ lines
in $A_{i+1}/B_{i-1}$, this gives $q$ choices for $B_i$,
and $q^\ell$ sequential choices in total for $B_1,B_2,\cdots,B_{\ell}$.

We next argue that the constant $c = [\ell]_q$.
Any flag $\AAA=(A_1,A_2,\ldots,A_{\ell})$
of length $\ell$
arises from products of the form
$(A_1) \cdot (B_1,B_2\ldots,B_{\ell})$
in which the flag $B_1 \subset B_2 \subset \cdots
\subset B_{\ell}$ has $A_1 \subseteq B_\ell$ (else, $(A_1) \cdot (B_1,B_2\ldots,B_{\ell})$ has length $\ell+1$, not $\ell$).  Letting $i_0$ be the smallest index for which $A_1 \subseteq B_{i_0}$, one finds that $1 \leq i_0 \leq \ell$.  Having fixed $i_0$,
the $B_i$ for $i$
in the range $i_0 \leq i \leq \ell$ 
are completely determined by 
$B_i = A_1+B_i=A_{i}$.
Meanwhile, for $i$ in the range $1 \leq i \leq i_0-1$, as in the argument for the constant $d=q^\ell$ above, one can sequentially choose each of $B_1,B_2,\ldots,B_{i_0-1}$ in $q$ ways so that they satisfy $A_1+B_i=A_{i+1}$.  This gives $q^{i_0-1}$ choices, which when summed over $i_0=1,2,\ldots,\ell$ gives $1+q+q^2+\cdots+q^{\ell-1}=[\ell]_q$ sequential choices in total.
\end{proof}

Lemma~\ref{x-multiplication} allows us to connect $R^{G}$ to the {\it Stirling} and {\it $q$-Stirling numbers}, briefly reviewed here.

\begin{definition}\rm
The classical {\it Stirling numbers of the second kind} 
$
(S(n,k))_{k,n=0,1,\ldots}
$
have two closely related families of $q$-analogues $S_q(n,k), \tilde{S}_q(n,k)$,
introduced by Carlitz \cite[\S4]{Carlitz} and
studied by many others, e.g., Cai, Ehrenborg and Readdy \cite{CaiReaddy}, Garsia and Remmel \cite{GarsiaRemmel}, Gould \cite{Gould}, deM\'edicis and Leroux \cite{deMedicisLeroux}, Milne \cite{Milne1, Milne2}, Sagan and Swanson \cite{SaganSwanson}, Wachs and White \cite{WachsWhite}, among others.
Using the notation\footnote{Notational conflicts are unavoidable. E.g., our $S_q(n,k), \tilde{S}_q(n,k)$ here
equal $\overline{S}[n,k], S[n,k]$,
respectively, in \cite{SaganSwanson}.} in \cite{Milne1}, 
all three are doubly-indexed triangles defined for $(n,k)$ with $n,k \geq 0$, having initial conditions that set them all equal to $1$ when $(n,k)=(0,0)$, and vanishing whenever $n+k \geq 1$ but either $k=0$ or $n=0$.  When both $n,k \geq 1$, they are then defined by these recursions:
\begin{equation}
\label{Stirling-recurrences}
\begin{array}{rccccl}
S(n,k)&=&S(n-1,k-1) &+& k \cdot S(n-1,k),\\
\tilde{S}_q(n,k)&=&\tilde{S}_q(n-1,k-1) &+& [k]_q \cdot \tilde{S}_q(n-1,k),\\
S_q(n,k)&=&q^{k-1}\cdot S_q(n-1,k-1) &+& [k]_q \cdot S_q(n-1,k).
\end{array}
\end{equation}
An easy induction using the recursion lets one check that, for all $n$ and $k$ one has the relation
$$
S_q(n,k)=q^{\binom{k}{2}} \tilde{S}_q(n,k),
$$
and for $n \geq 1$ one has
\begin{equation}
\label{boundary-q-Stirlings}
\begin{aligned}
S(n,1)&=S_q(n,1)=\tilde{S}_q(n,1)=1,\\
S(n,n)&=\tilde{S}_q(n,n)=1,\quad
S_q(n,n)=q^{\binom{n}{2}}.
\end{aligned}
\end{equation}
\end{definition}

\begin{remark}
Alternatively, one can consider $S(n,k), \tilde{S}_q(n,k), S_q(n,k)$
as change-of-basis matrices in the polynomial rings $\kk[t]$ with coefficients $\kk=\Z,\Z[q],\Z[q,q^{-1}]$, respectively. Consider the obvious ordered $\kk$-basis of $\kk[t]$ given by the powers $(t^n)_{n=0}^\infty=(1,t,t^2,\ldots)$,
versus these {\it $(q-)$falling factorial}
$\kk$-bases
$$
\begin{array}{rll}
(t)_n&:=t(t-1)(t-2)\cdots(t-(n-1))
 &\text{ in  }\Z[t],\\
(t)_{n,q}&:=t(t-[1]_q)(t-[2]_q)\cdots(t-[n-1]_q) &\text{ in  }\Z[q][t] \text{ or } \Z[q,q^{-1}][t].
\end{array}
$$
Then one has these change-of-basis formulas (see\footnote{The formulas as discussed by Milne in \cite[(1.14)]{Milne1} use the notation $[x]=\frac{y-1}{q-1}$ where $y=q^x$ is regarded as an indeterminate. To concord with notation and \eqref{Stirlings-as-change-of-bases} here, one should substitute $t=[x]=\frac{y-1}{q-1}$, so that $y=1+t(q-1)$.} Gould \cite[\S3]{Gould}, Milne
\cite[(1.14)]{Milne1}, \cite[(I.17)]{GarsiaRemmel}):

\begin{equation}
\label{Stirlings-as-change-of-bases}
\begin{array}{rll}
t^n&=\sum_k S(n,k) \cdot (t)_k &\text{ in }\Z[t],\\
t^n &=\sum_k \tilde{S}_q(n,k) \cdot (t)_{k,q} &\text{ in }\Z[q][t],\\
     &=\sum_k  S_q(n,k) q^{-\binom{k}{2}} \cdot (t)_{k,q} &\text{ in }\Z[q,q^{-1}][t].\\
\end{array}
\end{equation}
\end{remark}

We next show that $S(n,k), S_q(n,k)$ also mediate a natural change-of-basis within $R^G$.

\begin{cor}
\label{Stirling-corollary}
Let $\kk$ be a commutative ring with $1$,
and let $R=\kk M$ with $M=\freelrb_n$ or $\qfreelrb_n$.
Then the ($q$-)Stirling numbers $S(m,k), S_q(m,k)$ are the expansion
coefficients for the powers 
$\{ x^m \}_{m=0,1,\ldots,n}$ and $\{ (x^{(q)})^m \}_{m=0,1,\ldots,n}$ in the orbit-sum $\kk$-bases
$\{ x_k \}_{k=0,1,\ldots,n}, \{ x^{(q)}_k\}_{k=0,1,\ldots,n}$ of $R^G$:

$$
\begin{aligned}
x^m&=\sum_k S(m,k)\, x_k,\\
(x^{(q)})^m&=\sum_k S_q(m,k)\, x^{(q)}_k.
\end{aligned}
$$
Thus unitriangularity of $ \left \{S(m,k)\right\}$
shows $ \{ x^k \}_{k=0,1,\ldots,n}$
always gives a $\kk$-basis for $R^G$, while
triangularity of $\left \{S_q(m,k)\right\}$ shows 
$ \{ (x^{(q)})^k \}_{k=0,1,\ldots,n}$ are a $\kk$-basis for $R^G$ if and only if $q$ lies in $\kk^\times$.
\end{cor}
\begin{proof}
Both expansions follow by induction on $m$.
Here is the inductive step calcuation in the $q$-Stirling case, applying induction, Lemma~\ref{x-multiplication} and \eqref{Stirling-recurrences} for equalities (*),(**),(***), respectively:  
$$
\begin{aligned}
(x^{(q)})^m 
=x^{(q)} \cdot(x^{(q)})^{m-1}
&\overset{(*)}{=}x^{(q)} \cdot \sum_{k} S_q(m-1,k)\,\, x^{(q)}_k\\
&=\sum_{k} S_q(m-1,k)\,\, x^{(q)} \cdot x^{(q)}_k\\
&\overset{(**)}{=}\sum_{k} S_q(m-1,k)
\left( 
[k]_q x^{(q)}_k + q^k x^{(q)}_{k+1}
\right)\\
&=\sum_{k}
\left( 
[k]_q S_q(m-1,k) + q^{k-1} S_q(m-1,k-1)
\right) x^{(q)}_k
\overset{(***)}{=}\sum_{k}S_q(m,k) x^{(q)}_k.
\end{aligned}
$$
The $q$-expansion is invertible only when
$q$ lies in $\kk^\times$ due to triangularity
and $S_q(m,m)=q^{\binom{m}{2}}$.
\end{proof}

This leads to our answer for Question \ref{first-main-invariant-theory-question}.

\begin{thm}
\label{semisimplicity-theorem}
Let $\kk$ be any commutative ring with $1$, and $R=\kk M$ for either of the monoids $M=\freelrb_n, \qfreelrb_n$,
with symmetry groups $G=\symm_n, GL_n$.
If $M=\qfreelrb_n$, assume further that $q$ is in $\kk^\times$.
\begin{enumerate}
\item[(i)] 
The unique $\kk$-algebra map $
\kk[X] \overset{\gamma}{\longrightarrow} R$
defined by
$$
X  \longmapsto 
\begin{cases} 
x & \text{ if }M=\freelrb_n,\\
x^{(q)} & \text{ if }M=\qfreelrb_n,\\
\end{cases}
$$
induces an algebra isomorphism
$
\kk[X]/(f(X)) \cong R^G
$
where
$$
f(X):=
\begin{cases}
X(X-1)(X-2) \cdots (X-n)
& \text{ if }M=\freelrb_n,\\
X(X-[1]_q)(X-[2]_q) \cdots (X-[n]_q)
& \text{ if }M=\qfreelrb_n.\\
\end{cases}
$$
Hence $R^G$ is
commutative and generated by $x$ or $x^{(q)}$.

\item[(ii)]
If  $\kk$ is a field where $|G|$ is invertible, then $x$ or $x^{(q)}$
acts semisimply on any finite-dimensional $R^G$-module,
with eigenvalues contained in the lists
$$
\begin{cases}
0,1,2,\ldots,n & \text{ if }M=\freelrb_n,\\
[0]_q,[1]_q,[2]_q,\ldots,[n]_q
& \text{ if }M=\qfreelrb_n.\\
\end{cases}
$$
\end{enumerate}
\end{thm}

\begin{proof}
For assertion (i), note that
Lemma~\ref{x-multiplication}
shows that $x$ or $x^{(q)}$ acts on $R^G$ with
characteristic polynomial $f(X)$. 
Consequently, the kernel
of the algebra map $\kk[X] \overset{\gamma}{\longrightarrow}
R^G$ contains $f(X)$,
and $\gamma$ descends to a map on the quotient
$\kk[X]/(f(X)) \overset{\gamma}{\longrightarrow} R^G$.  
Since $f(X)$ is monic of degree $n+1$, the quotient
$\kk[X]/(f(X))$ has $\kk$-basis $(1,X,X^2,\ldots,X^n)$,
and Corollary~\ref{Stirling-corollary} shows that $\gamma$ maps this onto a $\kk$-basis of powers
$\{ x^k\}_{k=0}^n$ or $\{ (x^{(q)})^k\}_{k=0}^n$
for $R^G$.  Hence $\gamma$ is an
algebra isomorphism.

For assertion (ii), assume that 
$\kk$ is a field where the
roots of the characteristic polyonomial $f(X)$ of $x$ or $x^{(q)}$ 
acting on $R^G$ are all distinct.
This means that $f(X)$ must also
be the minimal polynomial for $x, x^{(q)}$ acting on
$R^G$, and that it 
acts semisimply in any finite dimensional $R^G$-module, with eigenvalues contained in that set of roots.  Lastly, note the groups $G$ have cardinalities
$$
|G|=\begin{cases}
|\symm_n|=n! & \text{ for }M=\freelrb_n,\\
|GL_n|=q^{n \choose 2}(q - 1)^n [n]!_q = (q^n-1)(q^n-q) (q^n-q^2)\cdots (q^n-q^{n-1}) & \text{ for }M=\qfreelrb_n,\\
\end{cases}
$$
where the $q$-factorial $[n]!_q$ is defined by
\begin{equation}
\label{q-factorial-definition}
[n]!_q:=[n]_q [n-1]_q \cdots [2]_q [1]_q.
\end{equation}
One can then check that the invertibility of $n!$ in $\kk$ and distinctness of $0,1,2,\ldots,n$ are both equivalent to $\kk$
having characteristic zero or a prime $p>n$, while
invertiblility of $|GL_n|$ in $\kk$ and distinctness of $[0]_q,[1]_q,[2]_q,\ldots,[n]_q$ are both equivalent to $\kk$
having characteristic zero or characteristic coprime to $q$ and to $[m]_q$ for $m=1,2,\ldots,n$.
\end{proof}

We close this section with some remarks on Brown's other $q$-analogues of $\freelrb_n$.

\begin{remark}
The analysis in Lemma~\ref{x-multiplication}
can be lifted to an analogous (and even simpler) computation in Brown's first $q$-analogue $\freelrb_{n,q}$.
Denoting the orbit sum $\kk$-basis in $\kk \freelrb_{n,q}$ by $y_0,y_1,\ldots,y_n$,  multiplication by the element $y:=y_1=\sum_{v \in (\F_q)^n \setminus \{\mathbf{0}\}} (v)$ acts on that basis as follows:
\begin{equation}
\label{x-multiplication-in-Browns-first-q-analogue}
y \cdot y_\ell = (q^\ell - 1)y_\ell + y_{\ell + 1}.
\end{equation}
Bearing in mind that the monoid surjection $\freelrb_{n,q} \overset{\pi}{\twoheadrightarrow} \qfreelrb_n$ described in Remark~\ref{three-q-analogues-warning-remark} has exactly 
$$
(q-1)(q^2-q) \cdots (q^\ell-q^{\ell-1})=(q-1)^\ell q^{\binom{\ell}{2}}
$$ 
preimages $(v_1,v_2,\ldots,v_\ell)$ for every flag $\AAA=(A_1,A_2,\ldots,A_\ell)$,
one can check that
\eqref{x-multiplication-in-Browns-first-q-analogue} maps under the linearization $\kk \freelrb_{n,q} \overset{\pi}{\twoheadrightarrow} \kk \qfreelrb_n$
to a formula consistent with the second formula in Lemma~\ref{x-multiplication}.
\end{remark}

\begin{remark}
It is also easy to check that Lemma~\ref{x-multiplication}
gives similar computations in the other monoids
$\overline{\freelrb}_n$
and $\overline{\freelrb}_{n,q}$
considered by Brown, discussed in Remark~\ref{three-q-analogues-warning-remark}. 
Specifically, in $\kk \overline{\freelrb}_n$, one has

\[
\overline{x} \cdot \overline{x}_\ell = \begin{cases}\ell \overline{x}_\ell + \overline{x}_{\ell + 1} &\text{ if $0 \leq \ell < n - 1$},\\ n \overline{x}_{n - 1} &\text{ if $\ell = n - 1$}\\\end{cases}
\]
and in $\kk \overline{\freelrb}_{n,q}$, one has
\[
\overline{x}^{(q)} \cdot \overline{x}^{(q)}_\ell = \begin{cases}[\ell]_q \overline{x}^{(q)}_\ell + q^\ell \overline{x}^{(q)}_{\ell + 1} &\text{ if $0 \leq \ell < n - 1$},\\ [n]_q\overline{x}^{(q)}_{n - 1} &\text{ if $\ell = n - 1$}.\\\end{cases}
\]
The point is that when one $\kk$-linearizes the monoid surjection
$\freelrb \rightarrow \overline{\freelrb}_n$
it maps $x_\ell \longmapsto \overline{x}_\ell$ for $i \leq n-2$, and maps $x_{n-1},x_n \longmapsto \overline{x}_{n-1}$. An analogous statement holds for  $\qfreelrb \rightarrow \overline{\freelrb}_{n,q}$.  One can then check that applying
these linearized surjections to
Lemma~\ref{x-multiplication} gives the above formulas.
\end{remark}

\section{Representation-theoretic preliminaries}
\label{representations-review-section}

Having answered  Question~\ref{first-main-invariant-theory-question}
by describing the structure of $R^G$,
the next few subsections collect and review 
some facts regarding representations of $G=\symm_n$ and $G=GL_n$ that will help us answer Question~\ref{second-main-invariant-theory-question} in Section \ref{main-results-section} on the structure of $R$, simultaneously as
an $R^G$-module and $G$-representation.

\subsection{Semisimplicity, filtrations, and eigenspaces}\label{sec:semisimplicityandfiltrations}

In what follows, we will be examining various modules $V$ over the monoid algebra $R=\kk M$ for the two monoids $M=\freelrb_n,\qfreelrb_n$, carrying
$\kk G$-module structures for
the automorphism groups $G=\symm_n, GL_n$.  In all cases, the $G$-actions on $R$ and $V$ will be {\it compatible} in the sense that
$$
g( r \cdot v )=g(r)\cdot g(v) 
\text{ for all }r \in R, v \in V, g \in G.
$$
Note that in this setting, $V$ carries commuting actions of
$R^G$ and of $\kk G$,
and we will wish to describe it simultaneously as a module over both.

Henceforth, assume that $\kk$ is a field in which $|G|$ is invertible, and take $V$ to be 
finite-dimensional over $\kk$.  This implies that $V$
is semisimple both as an $R^G$-module due to  Theorem~\ref{semisimplicity-theorem}(ii), and as a $\kk G$-module by Maschke's Theorem.

In order to answer Question \ref{second-main-invariant-theory-question}, we will utilize two important features of our setting:
\begin{enumerate}

    \item Semisimplicity implies that given a filtration by $R^G$-submodules
and $\kk G$-submodules $V_i$
$$
\{0\}=V_0 \subset V_1 \subset \cdots \subset V_r =V
$$
one actually has an $R^G$-module
and $\kk G$-module isomorphism
$$
V \cong \bigoplus_i V_i/V_{i-1}.
$$
This will play a crucial role in Section \ref{entire-semigroup-section} (and in particular our proof of Theorem \ref{thm:semigroupeigen}), where we will define filtrations on $\kk \freelrb_n$ and $\kk \qfreelrb_n$ that significantly simplify the analysis. \\

\item By Theorem~\ref{semisimplicity-theorem}(ii),
we have that $R^G$ is generated by
the single element $x$ or $x^{(q)}$, which acts
diagonalizably with certain eigenvalues $\lambda$ all lying in $\kk$. It follows that in order to understand the $R^{G}$ and $\kk G$-module structure of any module $V$, it suffices to decompose
the eigenspaces $\ker((x-\lambda)|_V)$
as $\kk G$-modules. 

Hence, we will answer Question \ref{second-main-invariant-theory-question} by describing the $j$-eigenspaces of $\kk \freelrb_n$ as $\symm_n$-representations and the  $[j]_q$-eigenspaces of $\kk \qfreelrb_n$ as $GL_n$- representations for $j = 0, 1, \ldots, n$.
\end{enumerate}

\subsection{Symmetric functions,  $\symm_n$-representations and unipotent $GL_n$-representations} 
\label{Frobenius-isomorphism-section}

We review here the relation between the {\it ring of symmetric functions} $\Lambda$ and representations of $\symm_n$;  see Sagan \cite{Sagan}, Stanley \cite{Stanley-EC2} as references, and for undefined terminology.  We then review the parallel story for R. Steinberg's {\it unipotent representations} of $GL_n$; see \cite[\S 4.2, 4.6, 4.7]{GrinbergReiner}
as a reference.

The {\it ring of symmetric functions} $\Lambda$ (of bounded degree, in infinitely many variables) may be viewed as a polynomial algebra $\Z[h_1,h_2,\ldots]=\Z[e_1,e_2,\ldots]$ where $h_n, e_n$ are the {\it complete homogeneous} and {\it elementary} symmetric functions of degree $n$.  One may view $\Lambda$ as a graded $\Z$-algebra $\Lambda=\bigoplus_{n=0}^\infty \Lambda^n$, which we wish to relate to the direct sum 
$$
\virtchars(\symm) := 
\bigoplus_{n=0}^\infty
\virtchars(\symm_n)
$$
where $\virtchars(\symm_n)$ denotes
the $\Z$-module of {\it virtual characters}
of $\symm_n$.  That is, $\virtchars(\symm_n)$ is the free $\Z$-module on the basis of
irreducible characters $\{\specht^\lambda\}$
indexed by the partitions $\lambda$ of $n$, or alternatively, the $\Z$-submodule of class functions on $\symm_n$ of the form $\chi-\chi'$ for genuine characters $\chi,\chi'$.  One makes $\virtchars(\symm)$ into a 
graded algebra via the {\it induction product} defined by
\begin{equation}
\label{parabolic-induction-defn}
\begin{array}{rcl}
\virtchars(\symm_{n_1}) \times 
\virtchars(\symm_{n_2}) &\longrightarrow& \virtchars(\symm_{n_1+n_2}) \\
(f_1,f_2) &\longmapsto& f_1*f_2:=(f_1 \otimes f_2)\uparrow_{\symm_{n_1} \times \symm_{n_2}}^{\symm_{n_1+n_2}}
\end{array}
\end{equation}
where $(-)\uparrow_H^G$ is the usual
{\it induction} of class functions on a subgroup $H$ to class functions on $G$.

For later use, we note that since $[\symm_{n_1+n_2}: \symm_{n_1} \times \symm_{n_2}]=\binom{n_1+n_2}{n_1}$, whenever $f_1, f_2$ are genuine characters, one has the formula for the degree of $f_1 * f_2$:
\begin{equation}
\label{induction-product-degree-formula}
\deg(f_1 * f_2) = \binom{n_1+n_2}{n_1} \deg(f_1) \deg(f_2).
\end{equation}

One then has the {\it Frobenius characteristic isomorphism} of $\Z$-algebras $\virtchars(\symm) 
\overset{\ch}{\longrightarrow} \Lambda$, mapping
$$
\begin{array}{rcl}
\virtchars(\symm) &
\overset{\ch}{\longrightarrow} &\Lambda\\ [0.05in]
\triv_{\symm_n} &\longmapsto & h_n\\ [0.05in]
\sgn_{\symm_n} & \longmapsto & e_n\\ [0.05in]
\specht^\lambda &\longmapsto & s_\lambda.
\end{array}
$$
Here $s_\lambda$ is the {\it Schur function}. For a composition $\alpha = \alpha_1, \alpha_2, \cdots, \alpha_\ell$, we use the standard shorthand  \[h_{\alpha}:= h_{\alpha_1}h_{\alpha_2}\cdots h_{\alpha_\ell}.\] 
For later use, note that one can express the 
regular representation 
$\kk \symm_n=\triv_{\symm_1} * \triv_{\symm_1} * \cdots *\triv_{\symm_1}$, implying
\begin{equation}
\label{regular-rep-Frobenius-image}
\ch\,\, \kk\symm_n = h_1^n = h_{1^n}.
\end{equation}

There is a parallel story for a certain subset of $GL_n$-representations.  Specifically, there is a collection of irreducible $GL_n$-representations $\{\specht_q^\lambda\}$, 
indexed by partitions $\lambda$ of $n$, which are the irreducible constituents occurring within the $GL_n$-permutation action on
the set $GL_n/B$ of {\it complete flags of subspaces} $\flags(V)$ in 
$V=(\F_q)^n$.  They were studied by R. Steinberg \cite{Steinberg-unipotents}, and are now called
the {\it unipotent characters} of $GL_n$.  
Denoting by $\virtchars(GL_n)$ the free $\Z$-submodule of the
class functions on $GL_n$ with unipotent characters
$\{\specht_q^\lambda\}$ as a basis, one can define the {\it parabolic} or {\it Harish-Chandra induction} product
on the direct sum $\virtchars(GL):=\bigoplus_{n=0}^\infty \virtchars(GL_n)$ as
follows:
$$
\begin{array}{rcl}
\virtchars(GL_{n_1}) \times 
\virtchars(GL_{n_2}) &\longrightarrow& \virtchars(GL_{n_1+n_2}) \\
(f_1,f_2) &\longmapsto&
f_1*f_2:=
\left( 
(f_1 \otimes f_2) \Uparrow_{GL{n_1} \times GL_{n_2}}^{P_{n_1,n_2}}
\right)
\uparrow_{P_{n_1,n_2}.}^{GL{n_1+n_2}}
\end{array}
$$
Here $P_{n_1,n_2}$ is the maximal parabolic subgroup of $GL_{n_1+n_2}$ setwise stabilizing the $\F_q$-span of the first $n_1$ standard basis vectors, and
$(-)\Uparrow_{GL{n_1} \times GL_{n_2}}^{P_{n_1,n_2}}$
is the {\it inflation} operation that creates a $GL{n_1} \times GL_{n_2}$-representation from a 
$P_{n_1,n_2}$-representation, by precomposing with the surjective
homomorphism 
$
P_{n_1,n_2} \twoheadrightarrow  GL_{n_1} \times GL_{n_2}
$
sending
$
\left[ \begin{smallmatrix} A & B \\ 0 & C \end{smallmatrix}\right] 
\longmapsto  
\left[ \begin{smallmatrix} A & 0 \\ 0 & C \end{smallmatrix}\right].
$
For later use, we note that since the inflation operation does not change the degree of a representation, and since 
$$
[GL_{n_1+n_2}: P_{n_1,n_2}]=\qbinom{n_1+n_2}{n_1}{q}
=\frac{[n_1+n_2]!_q}{[n_1]!_q [n_2]!_q},
$$
(with $[n]!_q$ as in \eqref{q-factorial-definition})
when $f_1, f_2$ are genuine characters, one has this degree formula for $f_1 * f_2$:
\begin{equation}
\label{parabolic-induction-product-degree-formula}
\deg(f_1 * f_2) = \qbinom{n_1+n_2}{n_1}{q} \deg(f_1) \deg(f_2).
\end{equation}

This parabolic induction operation turns out to make $\virtchars(GL)$ into an associative, commutative $\Z$-algebra.  One then has a $q$-analogue of the Frobenius isomorphism
$\virtchars(GL) \overset{\ch_q}{\longrightarrow} \Lambda$ sending\footnote{One might wonder which $GL_n$-character maps under $\ch_q$ to the elementary symmetric function $e_n$; it is
the {\it Steinberg representation}, in which $GL_n$ acts on the top homology of the {\it Tits building}, which is the
simplicial complex of flags of nonzero proper subspaces
in $(\F_q)^n$.}
$$
\begin{array}{rcl}
\virtchars(GL) &
\overset{\ch_q}{\longrightarrow} &\Lambda\\ [.05 in]
\triv_{GL_n} &\longmapsto & h_n\\  [.05 in]
\specht_q^\lambda &\longmapsto & s_\lambda.
\end{array}
$$
Note that the permutation representation 
$\kk[GL_n/B]$  of $GL_n$ on the complete flags can be expressed as $\triv_{GL_1} * \triv_{GL_1} * \cdots *\triv_{GL_1}$ and therefore one has this $q$-analogue of \eqref{regular-rep-Frobenius-image}:
\begin{equation}
\label{flag-rep-Frobenius-image}
\ch_q\,\, \kk[GL_n/B] = h_1^n = h_{1^n}.
\end{equation}

\subsection{($q$-)derangement numbers and representations}
\label{sec:derangement}
A central role in this story is played by the classical {\it derangement numbers} $d_n$, and the {\it $q$-derangement numbers} $d_n(q)$ of Wachs \cite{Wachs}:
\begin{equation}
\label{derangement-numbers}
\begin{array}{rcl}
d_n&:=&n!\sum_{k=0}^n \frac{(-1)^k}{k!}
=n!
\left(\frac{1}{0!} 
-\frac{1}{1!} 
+\frac{1}{2!}
-\frac{1}{3!} 
+\frac{1}{4!} 
- \cdots 
+\frac{(-1)^n}{n!}
\right)\\
d_n(q)&:=&[n]!_q \sum_{k=0}^n \frac{(-1)^k}{[k]!_q}q^{k \choose 2}.
\end{array}
\end{equation}
There are two well-known combinatorial models for $d_n$ counting permutations in $\symm_n$:
\begin{itemize}
    \item  {\it derangements}, which are the fixed-point free permutations, or 
    \item {\it desarrangements}, which are permutations $w=(w_1,w_2,\ldots,w_n)$ whose first ascent position $i$ with $w_i< w_{i+1}$ (using $w_{n+1}=n+1$ by convention) occurs for  an even position $i$.
\end{itemize}  
Wachs \cite{Wachs} and later D\'esarm\'enien and Wachs \cite{DesarmenienWachs} gave various interpretations for $d_n(q)$. In particular, $d_n(q)$ is still closely related to derangements and desarrangements. Letting $D_n, E_n$ denote the derangements and desarrangements in $S_n$, and defining the \textit{major index} statistic of a permutation $w=(w_1,\ldots,w_n)$
as $\text{maj}(\sigma)=\sum_{i:w_i > w_{i+1}} i$,
one has 
\[
d_n(q) = \sum_{\sigma \in D_n}q^{\text{maj}(\sigma)} = \sum_{\sigma \in E_n}q^{\text{maj}(\sigma^{-1})}.
\]

These $d_n, d_n(q)$ are the dimensions for a pair of representations
of $\symm_n$ and $GL_n$, which we call the {\it derangement representation} $\derangementrep_n$  and its (unipotent) $q$-analogue
$\derangementrep_n^{(q)}$.  Both have the same symmetric function image
$\derangement_n$ under the Frobenius maps $\ch$ and $\ch_q$,
a symmetric function with many equivalent descriptions.
For the reader's convenience, and for future use, we will compile these descriptions in Proposition \ref{derangement-symmetric-function-prop}, after first reviewing terminology.

Define for a permutation $w=(w_1,w_2,\ldots,w_n)$ in $\symm_n$ its {\it descent set} 
$$
\Des(w):=\{i \in \{1,2,\ldots,n-1\}:w_i > w_{i+1}\}.
$$
For example, $w=(6,3,5,2,1,4)$ has $\Des(w)=\{1,3,4\}$.
Note that the definition of a {\it desarrangement} given above may be rephrased as a permutation $w$ in $\symm_n$
for which the smallest element of $\{1,2,\ldots,n\}\setminus \Des(w)$ is even. Thus $w=(6,3,5,2,1,4)$ is a desarrangement,
since $\min(\{1,2,3,4,5,6\}\setminus\{1,3,4\}))=2$ is even.

Given a standard Young tableau $Q$ with $n$ cells written in English notation, its {\it descent set} is 
$$
\Des(w):=\{i \in \{1,2,\ldots,n-1\}: i+1 \text{ appears south and weakly west of }i\text{ in }Q \}.
$$
For example, 
\ytableausetup{smalltableaux}
$
Q=\begin{ytableau}
\ytableausetup{smalltableaux}
1&3\\
2&6\\
4\\
5
\end{ytableau}
$
has $\Des(Q)=\{1,3,4\}$.
Define a {\it desarrangement tableau} to be a standard Young tableau $Q$ with $n$ cells
for which the smallest element of $\{1,2,\ldots,n\}\setminus \Des(Q)$ is even.  Thus
the example tableau $Q$ given above is a desarrangement tableau.

Finally, for integers $n \geq 1$ and $D \subseteq [n]$, define
{\it Gessel's fundamental quasisymmetric function}
$$
L_{n,D}:=\sum_{\substack{1 \leq i_1 \leq i_2 \leq \cdots \leq i_n\\i_j < i_{j+1} \text{ if }j \in D}}
x_{i_1} x_{i_2} \cdots x_{i_n}
$$
which is a formal power series in $x_1,x_2,\ldots$, homogeneous of degree $n$.  For $w$ in $\symm_n$, let $\lambda(w)$ denote its cycle type partition of $n$.  For any partition $\lambda$ of $n$, the {\it higher Lie character} of Thrall \cite{Thrall}
or the {\it Gessel-Reutenauer symmetric function} $\higherLie
_\lambda$ (see \cite{gesselreutenauer}, \cite[\S6.6]{GrinbergReiner}, \cite[Exer. 7.89 ]{Stanley-EC2}) can be defined as
$$
\higherLie_\lambda:=\sum_{\substack{w \in \symm_n:\\ \lambda(w)=\lambda}}
L_{n,\Des(w)}.
$$

\begin{prop}
\label{derangement-symmetric-function-prop}
With the convention that $\derangement_0:=1$,
the following definitions of a sequence of 
symmetric functions
$\{\derangement_n\}_{n=0,1,2,\ldots}$
are all equivalent:
\begin{enumerate}
    \item [(A)]$\derangement_n=h_1 \derangement_{n-1} +(-1)^n e_n$ for $n \geq 1$;
    \item [(B)] $\derangement_n=\sum_{k=0}^n (-1)^k e_k \cdot h_{1^{n-k}}$;
    \item [(C)] $\derangement_n = h_{1^n}-\sum_{j=0}^{n-1} \derangement_j h_{n-j}$ (or equivalently, 
    $h_{1^n}=\sum_{j=0}^n \derangement_j h_{n-j}$) for $n \geq 1$;
    \item [(D)] $\derangement_n = \sum_Q s_{\lambda(Q)}$
    where $Q$ runs through the desarrangement tableaux of size $n$;
    \item [(E)] $\derangement_n = \sum_w L_{n,\Des(w)}$
    where $w$ runs through all desarrangements in $\symm_n$;
    \item [(F)] $\derangement_n = \sum_w L_{n,\Des(w)}$
    where $w$ runs through all derangements in $\symm_n$;
    \item[(G)] $\derangement_n = \sum_\lambda \higherLie_{\lambda}$ where $\lambda$ runs through all partitions of $n$ with no parts of size one.
\end{enumerate}
\end{prop}

We will mainly need definition (C) for $\derangement_n$.  However, we wish to point out that part (D) decomposes
$\derangement_n$ very explicitly into Schur functions, illustrated in the
table below for $n=0,1,2,3,4$.
\begin{equation}
\label{desarrangement-tableau-table}
\begin{tabular}{|c|c|c|}\hline
$n$ & desarrangement tableaux $Q$ &  $\derangement_n$ \\\hline\hline
$0$ &  $\varnothing$ & $1$\\\hline
$1$& \text{(none)} &$0$\\\hline
&  &\\
$2$ & 
$\ytableausetup{smalltableaux}
\begin{ytableau} 1\\2  
\end{ytableau}$
&$s_{(1,1)}$\\
&  &\\\hline
&  &\\
$3$ &
$\ytableausetup{smalltableaux}
\begin{ytableau} 1&3\\2  
\end{ytableau}$
&$s_{(2,1)}$\\
&  &\\\hline
&  &\\
$4$ &
$\ytableausetup{smalltableaux}
\begin{ytableau} 1\\2 \\3\\4  
\end{ytableau}$
\quad
$\ytableausetup{smalltableaux}
\begin{ytableau} 1&3\\2\\4  
\end{ytableau}$
\quad
$\ytableausetup{smalltableaux}
\begin{ytableau} 1&3\\2&4  
\end{ytableau}$\quad
$\ytableausetup{smalltableaux}
\begin{ytableau} 1&3&4\\2   
\end{ytableau}$\quad
&$s_{(1,1,1,1)} +
s_{(2,1,1)} +
s_{(2,2)} +
s_{(3,1)}
$\\
&  &\\\hline
\end{tabular}
\end{equation}

\begin{proof}[Sketch proof of Proposition~\ref{derangement-symmetric-function-prop}.]
We sketch some of the equivalences here. The equivalence of (A), (B) is  straightforward. Defining $\{\derangement_n\}$ by $(A)$, note they
satisfy definition (C) by induction on $n$:
$$
\begin{aligned}
\sum_{j=0}^n \derangement_j h_{n-j}
= \left(\sum_{j=1}^n \derangement_j h_{n-j} \right)+ h_n
&= \left( \sum_{j=1}^n (h_1 \derangement_{j-1} +(-1)^j e_j) \cdot h_{n-j} \right) + h_n \\
&= h_1 \sum_{j=1}^n \derangement_{j-1}h_{n-j}
\overset{(*)}{=} \sum_{j=0}^n (-1)^j e_j h_{n-j}\\ 
&\overset{(**)}{=} h_1 \cdot h_{1^{n-1}} + 0 
= h_{1^n}.
\end{aligned}
$$
Here equality (*) used $\sum_{j=0}^n (-1)^j e_j h_{n-j}=0$  for $n \geq 1$, and equality (**) used induction.  Consequently, (A) and (C) define the same sequence of polynomials $\{\derangement_n\}$, and so
 (A),(B),(C) coincide.

Defining $\{\derangement_n\}$ by the explicit sum $(D)$, let us check that they also satisfy the recursive definition (A) by induction on $n$.  In the base case $n=0$, both have $\derangement_0=1$, since the unique
(empty) tableau of size $0$ is a desarrangement tableau. In the inductive step, using the Pieri formula shows that $h_1 \cdot \derangement_{n-1}$ is the sum over all standard tableaux of size $n$ obtained from a desarrangement tableau $Q$ of size $n-1$ by adding $n$ in any corner cell.  This produces
all desarrangement tableaux of size $n$, except
the single column tableau $Q_0$ which
\begin{itemize}
    \item is produced for $n$ odd, but is {\it not} a desarrangement tableaux, and
   \item is not produced for $n$ even, but {\it is} a desarrangement tableau.
\end{itemize}
These exceptions are corrected by 
$(-1)^n e_n$ in the formula $\derangement_n=h_1 \derangement_{n_1}+(-1)^n e_n$ in (A).
Consequently, (A) and (D) define the same sequence of polynomials $\{\derangement_n\}$.

The equivalence of (D) and (E) uses two facts.  First, applying the Robinson-Schensted bijection to $w$ to obtain a pair of standard Young tableaux $(P,Q)$,
one has $\Des(w)=\Des(Q)$; see \cite[Lem. 7.23.1]{Stanley-EC2}. Thus $w$ is a desarrangement if and only if $Q$ is a desarrangement tableau\footnote{Our earlier examples $w=(6,3,5,2,1,4)$ and $Q$
also exemplify this, as $w \mapsto (P,Q)$ with \ytableausetup{smalltableaux}
$
Q=\begin{ytableau}
\ytableausetup{smalltableaux}
1&3\\
2&6\\
4\\
5
\end{ytableau}
$
and
$
P=\begin{ytableau}
\ytableausetup{smalltableaux}
1&4\\
2&5\\
3\\
6
\end{ytableau}
$.}. Second, $s_\lambda=\sum_P L_{\Des(P)}$ where $P$ runs over 
standard Young tableaux of shape $\lambda$ by \cite[Thm. 7.19.7]{Stanley-EC2}.

The equivalence of (E) and (F) was proven by D\'esarm\'enien and Wachs in \cite{FrenchDesarmenienWachs}, where they showed that both families of symmetric functions defined in (E),(F) satisfy the recursive definition (C).  Their proof also used the equivalence of (F) and (G) that follows from the definition of $\higherLie_\lambda$.
\end{proof}

Note that part (B) of
Proposition~\ref{derangement-symmetric-function-prop} 
generalizes the formulas in \eqref{derangement-numbers}, upon taking dimensions of the various representations and using \eqref{induction-product-degree-formula}, \eqref{parabolic-induction-product-degree-formula}.
Similarly, part (C) corresponds to the formulas
\begin{equation}
\label{chamber-space-dime-is-sum-of-eigenespace-dims}
\begin{aligned}
\dim_\kk \kk \symm_n = n!&=\sum_{j=0}^n d_{n-j}\binom{n}{j},\\
\dim_\kk \kk [GL_n/B] = [n]!_q&=\sum_{j=0}^n d_{n-j}(q)\qbinom{n}{j}{q},\\
\end{aligned}
\end{equation}
after taking into account \eqref{regular-rep-Frobenius-image}, \eqref{flag-rep-Frobenius-image}.

We conclude this section with some further historical remarks
and context on the derangement representations $\derangementrep_n$ and symmetric functions $\derangement_n$.

\begin{remark}
We are claiming no originality in
Proposition~\ref{derangement-symmetric-function-prop}.  As mentioned in its proof, the equivalence of (C),(E),(F),(G) is work of D\'esarm\'enien and Wachs \cite{FrenchDesarmenienWachs}.  
In \cite[Propositions 2.2, 2.1, 2.3]{ReinerWebb},  Webb and the third author noted that one can re-package their results to include part (D).  It was also noted there that the tensor product $\sgn \otimes\derangementrep_n$
 of $\derangementrep_n$ with the one-dimensional sign representation $\sgn$ of $\symm_n$, carries the same $\kk \symm_n$-module
    as the homology of the {\it complex of injective words} on $n$ letters.  Therefore, after tensoring with the sign character of $\symm_n$ or applying the fundamental involution $\omega$ on symmetric functions, parts (A), (C), (D) above correspond to \cite[Propositions 2.2, 2.1, 2.3]{ReinerWebb}. 
\end{remark}

\begin{remark} Hersh and the third author noted in \cite{HershReiner} that $\derangementrep_n$ occurs naturally
in the {\it representation stability and $FI$-module structure} (as in Church, Ellenberg and Farb \cite{CEF}) on the cohomology of the configuration space of $n$ labeled points in $\R^d$ for $d$ odd.  Specifically, $\derangementrep_n$ is
the $\kk\symm_n$-module on the subspace of {\it $FI$-module generators} for this cohomology, denoted $\widehat{\mathrm{Lie}}_n$ in \cite[Thms 1.2, 1.3]{HershReiner}.
\end{remark}

\begin{remark}
As hinted in the Introduction,
$\derangementrep_n$ also occurs as the $\kk \symm_n$-module on the kernel of 
    two shuffling operators on $\kk \symm_n$, both studied by Uyemura-Reyes: {\it random-to-top} shuffles \cite[\S 1.1.7, 3.2.2, 4.5.3]{Uyemura-Reyes} (also known as the {\it Tsetlin library}) and {\it random-to-random} shuffles \cite[Chap. 5]{Uyemura-Reyes};  see also Steinberg \cite[Prop. 14.5]{Steinberg} and Section~\ref{chamber-space-sec} below.
    More generally, Uyemura-Reyes described \cite[Thm. 4.1]{Uyemura-Reyes} the $\kk \symm_n$-module structure on the eigenspaces for {\it all} Bidigare-Hanlon-Rockmore  shuffling operators that carry $\symm_n$-symmetry.  Among these are random-to-top shuffles, whose eigenvalue multiplicities had previously been computed  by Phatarfod \cite{Phatarfod1991ONTM}, ignoring the $\kk\symm_n$-module structure.
    See also the discussion by Hanlon and Hersh in \cite[\S3]{HanlonHersh}
    and by Saliola, Welker and the third author in
    \cite[\S VI.9]{RSW}.
\end{remark}

\begin{remark}
In unpublished notes, Garsia \cite{Garsia} (see also Tian \cite{Tian}), studies the {\it top-to-random} shuffling operator, which is adjoint or transpose to the random-to-top operator. There he sketches a proof that its minimal polynomial is $X(X-1)(X-2)\cdots(X-n)$. The element $x$ acts as (rescaled) random-to-top on the chamber space of $\freelrb_n$ (see \cref{chamber-space-sec}). In light
of the fact that an operator and its transpose have the same minimal
polynomial, Garsia's sketch is closely related to the part of our proof of Theorem~\ref{semisimplicity-theorem}
dealing with $M=\kk\freelrb_n$.
\end{remark}

\section{Answering Question~\ref{second-main-invariant-theory-question}}
\label{main-results-section}

Our goal here is to answer Question~\ref{second-main-invariant-theory-question}, by describing the $\kk G$-module decompositions on the eigenspaces of $x, x^{(q)}$ as they act
on $\kk M$ for $M=\freelrb_n, \qfreelrb_n$.

Recall the ($\kk$-vector space) direct sum decompositions by length:
$$
\begin{array}{rcccl}
\kk \freelrb_n = 
 \bigoplus_{\ell = 0}^{n} \kk \freelrb_{n,\ell}
&\text{ where }&
\freelrb_{n,\ell}&:=\{\aa \in \freelrb_n: \ell(\aa)=\ell\},\\ [0.1in]
\kk \qfreelrb_n =
  \bigoplus_{\ell = 0}^{n} \kk \qfreelrb_{n,\ell}
&\text{ where }&
\qfreelrb_{n,\ell}&:=\{\AAA \in \qfreelrb_n: \ell(\AAA)=\ell\}.
\end{array}
$$
Following Brown \cite{BrownOnLRBs}, we call the monoid
elements of $\freelrb_{n,n}, \qfreelrb_{n,n}$ of maximum length {\it chambers}.  Their $\kk$-spans $\kk \freelrb_{n,n}$ and $\kk\qfreelrb_{n,n}$, which we call the {\it chamber spaces}, form submodules for the action of both the monoid algebras $\kk M$ and the group algebras $\kk G$. We first analyze the structure of these
chamber spaces in Section~\ref{chamber-space-sec},
and then use this to analyze the entire semigroup algebra $\kk M$ in Section~\ref{entire-semigroup-section}.

\subsection{The chamber spaces}
\label{chamber-space-sec}
 The chamber space $\kk \freelrb_{n,n}$ consists of all words of length $n$. Therefore, as a $\kk\symm_n$ module it is isomorphic to the  left regular-representation $\kk \symm_n$. Similarly, $\kk \qfreelrb_{n,n}$ has as a $\kk$-basis the set $\flags(V) = \{ \AAA=(A_1,\ldots,A_n) \}$
of all complete flags $A_1 \subset \cdots \subset A_{n-1} \subset A_n(=V)$, and is isomorphic to the coset representation of $GL_n$ on $\kk [GL_n/B]$.
 
We start with an old observation: multiplication by $x$ acts on $\freelrb_{n,n}$ as a (rescaled) version of the random-to-top operator on $\kk \symm_n$;
see for instance B. Steinberg \cite[Prop. 14.5]{Steinberg}. 

\begin{example} 
If $n=4$ and $w = (3,1,4,2)$ in $\freelrb_{4,4}$, then 
\[ 
\begin{aligned}
x \cdot w &=
((1)+(2)+(3)+(4)) \cdot (3,1,4,2) = 
(1,3,4,2)
+(2,3,1,4)
+(3,1,4,2)
+(4,3,1,2)
\end{aligned}
\]
which (after scaling by $\frac{1}{4}$) is
the result of random-to-top shuffling on $w$  as
an element of $\kk \symm_4$.
\end{example}

In this sense, the results in this section for the chamber space $\kk \freelrb_{n,n}$ are re-packaging previously mentioned results on random-to-top shuffling and the $\symm_n$-action on its eigenspaces, due to
Uyemura-Reyes \cite[Thm. 4.19]{Uyemura-Reyes}, building on the computation of Phatarfod \cite{Phatarfod1991ONTM} of the eigenvalue multiplicities.  On the other hand, as far as we aware, our results for the $q$-analogue $\kk \qfreelrb_{n,n}$ in Theorem \ref{chamber-eigenspaces-thm} are new.

We record here the action of $x^{(q)}$ on a complete
flag $\AAA$ in $V=(\F_q)^n$, using Definition~\ref{defn:q-freeLRB}:
$$
x^{(q)} \cdot \AAA=
\sum_{\text{lines }L \in V}
(L) \cdot \AAA = 
\sum_{\text{lines }L \in V}
(L, L+A_1, L+A_2, \ldots,L+A_{n-1},L+A_n)^{\wedge}.
$$

For $j = 0, 1, \ldots n$, we will write the $j$ and $[j]_q$-eigenspace of the chamber spaces
$\kk \freelrb_{n,n}, \kk \qfreelrb_{n,n}$ as
\[\ker\left( (x-j)|_{\kk \freelrb_{n,n}} \right), \hspace{3em} \ker\left( (x^{(q)}-[j]_q)|_{\kk \qfreelrb_{n,n}} \right). \]
In Theorem \ref{chamber-eigenspaces-thm} below, we relate these $j$ and $[j]_q$-eigenspaces to $\derangementrep_{n-j}$ and $\derangementrep^{(q)}_{n-j}$. Our proof depends crucially on Proposition~\ref{prop:operator}, Proposition~\ref{q-operator-proposition}, and  Lemma~\ref{Wachs-lemma} (all proved in Section~\ref{sec:constructingeigenvectors}) wherein we explicitly construct eigenvectors for the action of $x, x^{(q)}$ on the chamber spaces $\kk \freelrb_{n,n}, \kk \qfreelrb_{n,n}$ from the nullvectors of the same operators for smaller $n$. 

\begin{thm}
\label{chamber-eigenspaces-thm}
When $x, x^{(q)}$ act on $\kk \freelrb_{n,n}, \kk \qfreelrb_{n,n}$,
for each $j=0,1,2,\ldots,n$, their eigenspaces carry representations with
the same Frobenius map images
$$
\ch \ker\left( (x-j)|_{\kk \freelrb_{n,n}} \right)
= h_j \cdot \derangement_{n-j}
= \ch_q \ker\left( (x^{(q)}-[j]_q)|_{\kk \qfreelrb_{n,n}} \right).
$$
In other words, one has $\kk G$-module isomorphisms
$$
\begin{aligned}
\ker\left( (x-j)|_{\kk \freelrb_{n,n}} \right)
& \cong \triv_{\symm_j}* \derangementrep_{n-j} \\
\ker\left( (x^{(q)}-[j]_q)|_{\kk \qfreelrb_{n,n}} \right)
& \cong \triv_{GL_j}* \derangementrep^{(q)}_{n-j}.  
\end{aligned}
$$
\end{thm}
\begin{proof}

Lemma~\ref{Wachs-lemma} below exhibits $G$-equivariant
injections 
\begin{equation}
\label{eignespace-injections}
\begin{aligned}
\triv_{\symm_j} *\ker(x|_{\kk  \freelrb_{n-j,n-j}}) 
& \hookrightarrow \ker\left( (x-j)|_{\kk \freelrb_{n,n}} \right),\\
\triv_{GL_j} *\ker(x^{(q)}|_{\kk \qfreelrb_{n-j,n-j}}) 
& \hookrightarrow \ker\left( (x^{(q)}-[j]_q)|_{\kk \qfreelrb_{n,n}} \right).\\
\end{aligned}
\end{equation}

We now use facts proven by Phatarfod \cite{Phatarfod1991ONTM} for $q=1$ and by Brown \cite[\S 5.2]{BrownOnLRBs} 
for the $q$-analogue\footnote{A minor discrepancy here is that Brown analyzes the action of
$x^{(q)}$ not on the chamber space of $\kk\qfreelrb_n$ itself,
but rather on the chamber space of the quotient $\kk\overline{\freelrb}^{(q)}_n$ discussed in
Remark~\ref{three-q-analogues-warning-remark} above.  However,
just as Brown points out for $\freelrb_n$ and $\overline{\freelrb}_n$ in \cite[Remark, p. 888]{BrownOnLRBs},
the bijection 
$
(A_1, A_2,\ldots,A_{n-1},V) \mapsto
(A_1,A_2,\ldots,A_{n-1})
$
between chambers of $\freelrb^{(q)}_n$ and those of
$\kk\overline{\freelrb}^{(q)}_n$ will commute with both
the action of $GL_n$ and with multiplication by $x^{(q)}$.}:

$$
\begin{aligned}
\dim_\kk \ker(x|_{\kk\freelrb_{n,n}} )&=d_n,\\
\dim_\kk \ker(x^{(q)}|_{\kk \qfreelrb_{n,n}} )&=d_n(q).\\
\end{aligned}
$$
Hence the spaces on the left sides in \eqref{eignespace-injections}
have dimensions $d_{n-j}\binom{n}{j}$ and $d_{n-j}(q)\qbinom{n}{j}{q}$, respectively.
However, since eigenspaces for distinct eigenvalues are always linearly independent, and since
$$
\begin{aligned}
\kk \freelrb_{n,n} &\cong \kk \symm_n,\\
\kk \qfreelrb_{n,n} &\cong \kk[ GL_n/B],
\end{aligned}
$$
have dimensions $n!, [n]!_q$,
equations \eqref{chamber-space-dime-is-sum-of-eigenespace-dims}
imply that the injections in \eqref{eignespace-injections} must all be isomorphisms.

It also follows from the above analysis, or from Theorem~\ref{semisimplicity-theorem}(ii), that
$$
\begin{aligned}
\kk\freelrb_{n,n} &= \bigoplus_{j=0}^n \ker\left( (x-j)|_{\kk \freelrb_{n,n}} \right)\\
\kk\qfreelrb_{n,n} &= \bigoplus_{j=0}^n \ker\left( (x^{(q)}-[j]_q)|_{\kk \qfreelrb_{n,n}} \right).
\end{aligned}
$$
Then using \eqref{regular-rep-Frobenius-image}, \eqref{flag-rep-Frobenius-image} and comparing with Proposition~\ref{derangement-symmetric-function-prop}(C), the theorem follows.
\end{proof}

\subsubsection{Constructing eigenvectors from nullvectors: proof of
Lemma~\ref{Wachs-lemma}.}\label{sec:constructingeigenvectors}
The goal of this subsection is to prove Lemma~\ref{Wachs-lemma}.  It relies on parallel constructions\footnote{The third author is grateful to Michelle Wachs for explaining to him the $\kk \freelrb_n$ version of this construction (the operator $\Psi_U$) in 2002, in the context of random-to-top shuffling.} of eigenvectors for $x, x^{(q)}$ acting on the spaces $\kk \freelrb_{n,n},
\kk \qfreelrb_{n,n}$ from nullvectors for the same operators for
smaller $n$.

\begin{definition}\rm
Let $[n]:=\{1,2,\ldots,n\}$, and
fix a $j$-element subset $U$ 
of $\{1,2,\ldots,n\}$.
Let $\symm_{[n]\setminus U}$ denote the permutations
$\aa=(a_1,a_2,\ldots,a_{n-j})$
of the complementary subset $[n] \setminus U$,
written in one-line notation.  On the $\kk$-vector space 
$\kk[\symm_{[n]\setminus U}]$ having these permutations
as a $\kk$-basis, define two maps $\Psi_U,\Phi_U: \kk[\symm_{[n]\setminus U}]
\longrightarrow \kk[\symm_n]$
by extending these rules $\kk$-linearly:
$$
\begin{aligned}
\Psi_U(\aa)&:=
\displaystyle\sum_{ \bb \in \symm_U} 
(b_1,b_2,\ldots,b_j,a_1,a_2,\ldots,a_{n-j})\\
\Phi_U(\aa)&:=
\displaystyle\sum_{ \bb \in \symm_U} 
(a_1,b_1,b_2,\ldots,b_j,a_2,\ldots,a_{n-j})
\end{aligned}
$$
where the summation indices $\bb$
run over all permutations
$\bb=(b_1,b_2,\ldots,b_j)$
in $\symm_U$.
\end{definition}

\begin{example}
Let $n=5$ and $U=\{4,5\}$.  Then 
$$
\begin{aligned}
\Psi_{U}((1,2,3)) &= (\mathbf{4,5},1,2,3) + (\mathbf{5,4},1,2,3),\\
\Phi_{U}((1,2,3)) &= (1,\mathbf{4,5},2,3) + (1,\mathbf{5,4},2,3).\\
\end{aligned}
$$
\end{example}

To state the next proposition, introduce
for $U \subseteq [n]$ the {\it free left-regular band $\freelrb_U$ on $U$}, having an obvious isomorphism $\freelrb_U \cong \freelrb_j$ if $j=|U|$.  Also let 
$
x_U:=\sum_{i \in U} (i)
$
inside $\kk \freelrb_U$. 

\begin{prop}\label{prop:operator}
Fix a $j$-element subset $U$ of $[n]$ and
a permutation $\aa$ in $\symm_{[n]\setminus U}$.
Then 
$$
x \cdot \Psi_{U}(\aa) 
= j \cdot \Psi_{U}(\aa) 
 + \Phi_{U}\left( x_{[n]\setminus U} \cdot \aa \right).
$$
Consequently, if $v$ in $\kk \freelrb_{[n]\setminus U,n-j}$
has $x_{[n]\setminus U} \cdot v=0$, then
$\Psi_{U}(v)$ is a $j$-eigenvector for $x$ 
on $\kk \freelrb_{n,n}$:
$$
x \cdot \Psi_{U}(v) = j \cdot \Psi_{U}(v). 
$$
\end{prop}
\begin{proof}
One can calculate as follows:
$$
\begin{array}{rcccl}
x \cdot \Psi_{U}(\aa) 
= \displaystyle\sum_{i=1}^n (i) \cdot \Psi_{U}(\aa)
& = &\displaystyle\sum_{i \in U } (i) \cdot \Psi_{U}(\aa)&
+ &\displaystyle\sum_{i \in [n] \setminus U} (i) \cdot \Psi_{U}(\aa)\\
& & & & \\
&=&j \cdot \Psi_{U}(\aa) & + & \Phi_{U}\left( x_{[n]\setminus U} \cdot \aa \right)
\end{array}
$$
where we explain here the two substitutions in the last equality.  The fact that the left sum equals
$j \cdot \Psi_{U}(\aa)$
follows from the last equation $x \cdot x_j = j \cdot x_j$ in 
Lemma~\ref{x-multiplication}
applied to $\kk \freelrb_{U} \cong \kk \freelrb_j$.
The fact that the right sum is 
$\Phi_{U}\left( x_{[n]\setminus U} \cdot \aa \right)$
follows via direct calculation from the definitions.
\end{proof}

We next introduce two $q$-analogous maps $\Psi_U^{(q)}, \Phi_U^{(q)}$.

\begin{definition}\rm
Fix $U$ a $j$-dimensional $\F_q$-linear subspace of $V=(\F_q)^n$. Let $\flags(V/U)$ denote the set of maximal flags in the quotient space $V/U$
$$
\AAA=(A_1,A_2\ldots,A_{n-j-1},\underbrace{A_{n-j}}_{=V/U}).
$$
On the space 
$\kk[\flags(V/U)]$ with these flags
as $\kk$-basis, define
maps $\Psi^{(q)}_U,\Phi^{(q)}_U: \kk[\flags(V/U)]
\longrightarrow \kk[\flags(V)]$
by extending the following rules $\kk$-linearly:
$$
\begin{aligned}
\Psi^{(q)}_U(\AAA)&:=
\displaystyle\sum_{ \BBB \in \flags(U)} 
(B_1,B_2,\ldots,B_{j-1},U,
A_1+U,A_2+U,\ldots,A_{n-j-1}+U,V),\\
\Phi^{(q)}_U(\AAA)&:=
\displaystyle
\sum_{\substack{\text{lines }L:\\L \subset U+A_1,\\ L \not\subset U}} \quad
\sum_{ \BBB \in \flags(U)} 
(L,L+B_1,\ldots,L+B_{j-1},\underbrace{L+U}_{=L+U+A_1},
L+U+A_2,\ldots,L+U+A_{n-j-1},V),
\end{aligned}
$$
where the summation indices $\BBB$
run over all complete flags
$\BBB=(B_1,\ldots,B_{j-1},U)$
in $\flags(U)$.
\end{definition}

To state the next proposition, introduce
for any $\F_q$-vector space $U$ of dimension $j$ 
the monoid $\qfreelrb_U \cong \qfreelrb_j$ by identifying $U \cong \F_q^j$.  Also introduce the element of the monoid algebra $\kk \qfreelrb_U$
$$
x^{(q)}_U:=\sum_{\text{lines }L\text{ in }U} (L).
$$

\begin{prop} 
\label{q-operator-proposition}
For a $j$-dimensional subspace $U$ of $V=(\F_q)^n$ and
complete flag $\AAA$ in $\flags(V/U)$,
$$
x^{(q)} \cdot \Psi^{(q)}_{U}(\AAA) 
= [j]_q \cdot \Psi^{(q)}_{U}(\AAA) 
 +  \Phi^{(q)}_{U}\left( x^{(q)}_{V/U} \cdot \AAA \right).
$$
Hence if $v$ in $\kk \qfreelrb_{V/U,n-j}$
has $x^{(q)}_{V/U} \cdot v=0$, then
$\Psi^{(q)}_{U}(v)$ is a $[j]_q$-eigenvector for $x^{(q)}$ 
on $\kk \qfreelrb_{n,n}$:
$$
x^{(q)} \cdot \Psi^{(q)}_{U}(v) = [j]_q \cdot \Psi^{(q)}_{U}(v).
$$
\end{prop}
\begin{proof}
One can calculate as follows:
$$
\begin{array}{rcccl}
x^{(q)} \cdot \Psi^{(q)}_{U}(\AAA) 
= \displaystyle\sum_{\substack{\text{lines }L \\\text{ in }V}} (L) \cdot \Psi^{(q)}_{U}(\AAA)
& = &\displaystyle\sum_{\substack{\text{lines }L\\\text{ in }U}} (L) \cdot \Psi^{(q)}_{U}(\AAA)&
+ &\displaystyle\sum_{\substack{\text{lines }L\\\text{ not in }U}} (L) \cdot \Psi^{(q)}_{U}(\AAA)\\
& & & & \\
&=&[j]_q \cdot \Psi^{(q)}_{U}(\AAA) & + & \Phi^{(q)}_{U}\left( x^{(q)}_{V/U} \cdot \AAA \right)
\end{array}
$$
where we explain here the two substitutions in the last equality. The fact that the left sum equals
$[j]_q \cdot \Psi^{(q)}_{U}(\AAA)$
follows from the last equation $x^{(q)} \cdot x^{(q)}_j = [j]_q \cdot x^{(q)}_j$ in 
Lemma~\ref{x-multiplication}
applied to $\kk \qfreelrb_{U} \cong \kk \qfreelrb_j$.  To check the substitution made for
the the right sum, one calculates directly that
$$
\begin{aligned}
&\Phi^{(q)}_{U}\left( x^{(q)}_{V/U} \cdot \AAA \right)\\
& =\sum_{\substack{\text{lines }\overline{L}\\\text{ in }V/U}}
 \Phi^{(q)}_{U}\left(  (\overline{L}) \cdot \AAA \right)\\
& =\sum_{\substack{\text{lines }\overline{L}\\\text{ in }V/U}}
 \Phi^{(q)}_{U}\left( 
  (\overline{L},\overline{L}+A_1,\overline{L}+A_2,\ldots,\overline{L}+A_{n-j-1},V/U)^{\wedge}
 \right)\\
&=\sum_{\substack{\text{lines }\overline{L}\\\text{ in }V/U}} \quad 
\sum_{
\substack{
\text{lines }L\text{ in }V:\\
L \subset U+\overline{L}\\
L \not\subset U}} \quad 
\sum_{\BBB \in \flags(U)}
  (L,L+B_1,\ldots,L+B_{j-1},\underbrace{L+B_j}_{=L+U},L+U+A_1,\ldots,L+U+A_{n-j-1},V)^{\wedge}\\
& =
\sum_{\substack{
\text{lines }L\text{ in }V:\\
L \not\subset U}} \quad 
\sum_{\BBB \in \flags(U)} \quad 
  (L) \cdot (B_1,\ldots,B_{j-1},\underbrace{B_j}_{=U},A_1+U,A_2+U,\ldots,A_{n-j-1}+U,V)
=
\sum_{\substack{\text{lines }L\\\text{ not in }U}} (L) \cdot \Psi^{(q)}_{U}(\AAA).\qedhere
\end{aligned}
$$
\end{proof}

We are at last ready to prove Lemma~\ref{Wachs-lemma}.

\begin{lemma}
\label{Wachs-lemma}
With our usual notation of $G=\symm_n, GL_n$ acting on $\kk M$
for $M=\freelrb_n, \qfreelrb_n$, one has $G$-equivariant
injections for $j=0,1,2,\ldots,n$:
$$
\begin{aligned}
\triv_{\symm_j} * \ker(x|_{\kk \freelrb_{n-j,n-j}})  
& \hookrightarrow \ker\left( (x-j)|_{\kk \freelrb_{n,n}} \right),\\
\triv_{GL_j} * \ker(x^{(q)}|_{\kk \qfreelrb_{n-j,n-j}})  
& \hookrightarrow \ker\left( (x^{(q)}-[j]_q)|_{\kk \qfreelrb_{n,n}} \right).\\
\end{aligned}
$$
\end{lemma}
\begin{proof}
We give the proof for $\qfreelrb_n$; the proof for $\freelrb_n$ is analogous, but easier.

For each $j$-dimensional subspace $U$ of $V=(\F_q)^n$,
define a subspace $E(U)$ of $\kk \qfreelrb_{n,n}$ as the image under $\Psi^{(q)}_{U}$ of the nullspace for $x^{(q)}=x_{V/U}^{(q)}$
acting on $\kk \qfreelrb_{V/U,n-j} \cong \kk \qfreelrb_{n-j,n-j}$:
$$
E(U):=\Psi^{(q)}_{U}\left( \ker
x^{(q)}|_{\kk \qfreelrb_{V/U,n-j}} 
\right).
$$
According to Proposition~\ref{q-operator-proposition},
each $E(U)$ is a subspace of the $[j]_q$-eigenspace 
$\ker((x^{(q)}-[j]_q)|_{\kk \qfreelrb_{n,n}})$.  Note that vectors in
$E(U)$ are sums of complete flags $\AAA=(A_1,\ldots,A_n)$ that pass through $A_j=U$,
and hence for $U \neq U'$, they are supported on basis elements of $\kk\qfreelrb_{n,n}$
indexed by disjoint sets of complete flags.
Therefore the subpsace sum of all $E(U)$ is a direct sum $\bigoplus_{U} E(U)$
inside this $[j]_q$-eigenspace for $x$. It only remains to produce an isomorphism of
$GL_n$-representations 
\begin{equation}
\label{Webb-lemma-application}
\bigoplus_{U} E(U) 
\quad \cong \quad
\triv_{GL_j} * \ker(x^{(q)}|_{\kk \qfreelrb_{n-j,n-j}}).
\end{equation}
Recall $GL_n$ acts transitively on $j$-subspaces $U$. Fix
the particular subspace $U_0$
spanned by the first $j$ standard basis vectors in $V=(\F_q)^n$,
whose $GL_n$-stabilizer is the maximal parabolic subgroup $P_{j,n-j}$.  It follows (see, e.g.,
 Webb \cite[Prop. 4.3.2]{Webb}), that $\bigoplus_{U} E(U)$ 
 carries the $GL_n$-representation induced from $P_{j,n-j}$ acting on $E(U_0)$.  However, because elements in $E(U_0)$ are
 supported on flags $\AAA$ in $E(U_0)$ that
 all pass through $A_j=U_0$, this $P_{j,n-j}$-action is 
 inflated through
 the surjection $P_{j,n-j} \twoheadrightarrow GL_j \times GL_{n-j}$.  Furthermore, the definition of 
 $\Psi^{(q)}_{U_0}(-)$ as a symmetrized 
 sum over complete flags in $U_0$ shows that $GL_j$ fixes
 elements of $E(U_0)$ pointwise, while elements of $GL_{n-j}$ act
 as they do on $\ker(x^{(q)}|_{\kk \qfreelrb_{n-j,n-j}})$.
 Comparing with \eqref{parabolic-induction-defn} proves the desired isomorphism
 \eqref{Webb-lemma-application}.
\end{proof}

\subsection{The entire semigroup algebra}\label{entire-semigroup-section}

Having described the eigenspaces of the chamber spaces $\kk \freelrb_{n,n}$ and $\kk \qfreelrb_{n,n}$ as $G$-representations, we now turn to the entire semigroup algebras $\kk \freelrb_n$ and $\kk \qfreelrb_n$.

Our strategy here will be to introduce filtrations on $\kk \freelrb_n$, $\kk \qfreelrb_n$, and study the action of $x$ and $x^{(q)}$ on the associated graded modules with respect to these filtrations. (Recall from the discussion in Section~\ref{sec:semisimplicityandfiltrations} that by semisimplicity, this is an equivalent way to understand the $R^G$ and $\kk G$-module structures on $\kk \freelrb_n$, $\kk \qfreelrb_n$.)

Recall that for $\aa \in \freelrb_n$ and $\AAA \in \qfreelrb_n$ the length of $\aa$ is $\ell(\aa)$, and the length of $\AAA$ is $\ell(\AAA)$. 

\begin{definition}\label{def:filtration} \rm
Define
\[ \begin{aligned}
\kk \freelrb_{n,\geq \ell} :=& \span_{\kk} \{ \aa \in \freelrb_n: \ell(\aa) \geq \ell \}, \\ 
 \kk \qfreelrb_{n,\geq \ell} :=& \span_{\kk} \{ \AAA \in \qfreelrb_n: \ell(\AAA) \geq \ell \}.     \end{aligned}
\]
In other words, $\kk \freelrb_{n,\geq \ell}$ and $\kk \qfreelrb_{n,\geq \ell}$ are the $\kk$-spans of the monoid elements of length at least $\ell$.
\end{definition}

We then introduce filtrations
$\{ \kk \freelrb_{n,\geq \ell}\}_{\ell= 0,1,\dots,n,n+1}$ and $\{ \kk \qfreelrb_{n,\geq \ell}\}_{\ell= 0,1,\dots,n,n+1}$:
\begin{equation}
\label{the-filtrations}
\begin{aligned}
\{0\} &= \kk\freelrb_{n,\geq n+1} 
\subset  \kk\freelrb_{n, \geq n}
\subset \cdots 
\subset \kk\freelrb_{n, \geq 1}
\subset \kk\freelrb_{n, \geq 0} = \kk\freelrb_n, \\
\{0\} &= \kk\qfreelrb_{n,\geq n+1} 
\subset  \kk\qfreelrb_{n, \geq n}
\subset \cdots 
\subset \kk\qfreelrb_{n, \geq 1}
\subset \kk \qfreelrb_{n, \geq 0} = \kk\qfreelrb_n.
\end{aligned}
\end{equation}
Since $\ell(\aa \cdot \bb) \geq \ell(\bb)$,
it is easily seen that each $\kk \freelrb_{n,\geq \ell}$ is
a $\kk\freelrb_n$-submodule, and a $\kk\symm_n$-submodule. Analogously, $\kk \qfreelrb_{n,\geq \ell}$ is
a $\kk\qfreelrb_n$-submodule, and a $\kk GL_n$-submodule.

Recall that for $U \subset [n]$ of size $j$ one has  
$
\freelrb_U \cong \freelrb_j
$
and $x_{U} = \sum_{i \in U} (i)$. Analogously, recall that for $U$ a $j$-dimensional subspace of $V$, one has $\qfreelrb_U \cong \qfreelrb_j$ and 
\[ x^{(q)}_U = \sum_{\substack{\textrm{lines $L$}\\\text{ in $U$}}} (L). \]
Both $\kk\freelrb_U$ and $\kk\qfreelrb_U$
have $\kk$-vector space direct sum decompositions defined by length of words, so that one can identify $\kk\freelrb_{U,\ell} \cong \kk\freelrb_{j,\ell}$ and $\kk\qfreelrb_{U,\ell} \cong \kk\qfreelrb_{j,\ell}$ for $\ell = 0, 1, \ldots j$.

As $\kk$-vector spaces, one has a direct sum decomposition for the filtration factors
\begin{equation}
\begin{aligned}
\label{vector-space-direct-sum}
\kk \freelrb_{n,\geq \ell}
/\kk \freelrb_{n,\geq \ell+1}
&= \bigoplus_{\substack{U \subseteq \{1,2,\ldots,n\}:\\|U|=\ell}} \overline{\kk \freelrb_{U,\ell}} \\
\kk \qfreelrb_{n,\geq \ell}
/\kk \qfreelrb_{n,\geq \ell+1}
&= \bigoplus_{\substack{\F_q\text{-subspaces }U \subseteq (\F_q)^n:\\ \dim(U)=\ell}} \overline{\kk \qfreelrb_{U,\ell}}
\end{aligned}
\end{equation}
where $\overline{\kk \freelrb_{U,\ell}}$ and $\overline{\kk \qfreelrb_{U,\ell}}$ denote the image of the subspaces 
$\kk \freelrb_{U,\ell}$ and $\kk \qfreelrb_{U,\ell}$ within the quotient
on the left.  The next proposition is
a simple but crucial observation about these
summands in \eqref{vector-space-direct-sum},
used in the proof of Theorem \ref{thm:semigroupeigen}.

\begin{prop}
\label{summands-as-submodules-prop}
Consider the summands on the right sides of
\eqref{vector-space-direct-sum}.
\begin{itemize}
\item
Each $\overline{\kk \freelrb_{U,\ell}}$
is a $\kk\freelrb_n$-submodule of
$\kk \freelrb_{n,\geq \ell}
/\kk \freelrb_{n,\geq \ell+1}$,
annihilated by $(j)$ for $j \not\in U$.
\item
Each $\overline{\kk \qfreelrb_{U,\ell}}$
is a $\kk\qfreelrb_n$-submodule of
$\kk \qfreelrb_{n,\geq \ell}
/\kk \qfreelrb_{n,\geq \ell+1}$,
annihilated by $(L)$ for lines $L \not\subset U$. 
\end{itemize}

Consequently, one has
$$
\begin{array}{rccll}
x \cdot \overline{\aa}
&=&x_U \cdot \overline{\aa}
&\text{ for }
\overline{\aa} \text{ in }\overline{\kk \freelrb_{U,\ell}},\\
x^{(q)} \cdot \overline{\AAA}
&=&x^{(q)}_U \cdot \overline{\AAA}
&\text{ for }\overline{\AAA} \text{ in }\overline{\kk \qfreelrb_{U,\ell}}.
\end{array}
$$
\end{prop}

\begin{proof}[Proof by example.] Consider $n=3$ with $\ell=2$ and $U = \{ 1, 2 \}$. 
Then working in the quotient
$\overline{\kk \freelrb_{U, 2}}$, because
$3 \not\in U$,
the element $(3)$ of $\kk \freelrb_3$
will annihilate the element
$\overline{(1,2)}$ of $\kk \freelrb_{3,\geq 2}
/\kk \freelrb_{3,\geq 3}$. One has
$$
(3) \cdot \overline{(1,2)}=\overline{(3,1,2)}=0
\quad \text{ in }\kk \freelrb_{3,\geq 2}
/\kk \freelrb_{3,\geq 3}
$$
because $\ell(3,1,2)=3 > 2=\ell$. Thus $x = (1) + (2) + (3)$ acts on $\overline{(1,2)}$ as follows:
\[  
\begin{aligned}
x \cdot \overline{(1,2)}=
((1) + (2) + (3)) \cdot \overline{(1,2)}
= \overline{(1,2)} + \overline{(2,1)} + \overline{(3,1,2)}
=\overline{(1,2)}+\overline{(2,1)}
= x_U \cdot \overline{(1,2)}.
\end{aligned}
\]
The proof for $\qfreelrb_n$ is analogous:  one has $\ell( (L) \cdot \AAA) > \ell(\AAA)=\ell$ for lines $L \not\subset U$ and
$\AAA \in \qfreelrb_{U,\ell}$.
%
%
\end{proof}

We now prove our main result of this section,
encompassing Theorem~\ref{thm:semigroupeigen} from the Introduction.

\begin{thm}
\label{precise-version-of-semigroupeigen}
Let $\kk$ be a field in which $|G|$ is invertible. Then $x, x^{(q)}$ act diagonalizably on $\kk \freelrb_{n}, \kk \qfreelrb_{n}$, and
for each $j=0,1,2,\ldots,n$, their eigenspaces carry representations with
the same Frobenius map images
$$
\ch \ker\left( (x-j)|_{\kk \freelrb_n} \right)
= \sum_{\ell=j}^n \ h_{(n-\ell,j)} \cdot \derangement_{\ell-j}
= \ch_q \ker\left( (x^{(q)}-[j]_q)|_{\kk \qfreelrb_n} \right).
$$
In other words, one has $\kk G$-module isomorphisms
$$
\begin{aligned}
\ker\left( (x-j)|_{\kk \freelrb_n} \right)
& \cong \bigoplus_{\ell=j}^n
 \triv_{\symm_{n-\ell}} * \triv_{\symm_j} * \derangementrep_{\ell-j}, \\
\ker\left( (x^{(q)}-[j]_q)|_{\kk \freelrb_n^{(q)}} \right)
& \cong \bigoplus_{\ell=j}^n   \triv_{GL_{n-\ell}} * \triv_{GL_j} * \derangementrep^{(q)}_{\ell-j}.
\end{aligned}
$$
\end{thm}

\begin{proof}
The filtrations in \eqref{the-filtrations} show that
\begin{equation}
\begin{aligned}
\ker\left( (x-j)|_{\kk \freelrb_n} \right)
&\cong
\bigoplus_{\ell=0}^n 
\ker \left( (x-j)|_{\kk \freelrb_{n,\geq \ell}
/\kk \freelrb_{n,\geq \ell+1}} \right),\\
\ker\left( (x^{(q)}-[j]_q)|_{\kk \qfreelrb_n} \right)
&\cong
\bigoplus_{\ell=0}^n 
\ker \left( (x^{(q)}-[j]_q)|_{\kk \qfreelrb_{n,\geq \ell}
/\kk \qfreelrb_{n,\geq \ell+1}} \right).
\end{aligned}
\end{equation}
It remains to analyze each summand on the right. 

We have seen that \eqref{vector-space-direct-sum} is
also a direct sum decomposition as $\kk M$-modules for $M = \kk\freelrb_n, \kk \qfreelrb$.  
For $G = \symm_n, GL_n$, the action of $\kk M$
and $\kk G$ on
both sides in \eqref{vector-space-direct-sum} commute. 

In the case of $M = \freelrb_n$, this leads to the following equalities
and isomorphisms of $\kk \symm_n$-modules, explained below. Let $U_0:= \{ 1, 2, \ldots, \ell \}$. Then
\begin{equation*}
\label{eigenspace-space-direct-sum}
\begin{aligned}
\ker \left( (x-j)|_{\kk \freelrb_{n,\geq \ell}
/\kk \freelrb_{n,\geq \ell+1}} \right)
&\overset{(i)}{=} \bigoplus_{\substack{U \subseteq \{1,2,\ldots,n\}:\\|U|=\ell}} 
\ker( (x-j)|_{\overline{\kk \freelrb_{U,\ell}}} )\\
&\overset{(ii)}{=} \bigoplus_{\substack{U \subseteq \{1,2,\ldots,n\}:\\|U|=\ell}} 
\ker( (x_{U}-j)|_{\overline{\kk \freelrb_{U,\ell}}} )\\
&\overset{(iii)}{\cong} 
\triv_{\symm_{n-\ell}} * \ker\left( (x_{U_{0}}-j )|_{\kk \freelrb_{\ell,\ell}}\right)  \\
&\overset{(iv)}{\cong} \begin{cases}
0 &\text{ if $\ell < j$}\\
\triv_{\symm_{n-\ell}} *  \triv_{\symm_{j}} * \derangementrep_{\ell -j} &\text{ if $\ell \geq j$.}
\end{cases}
\end{aligned}
\end{equation*}     

\begin{itemize}
    \item Equality (i) is the restriction of
the $\kk\symm_n$-module isomorphism \eqref{vector-space-direct-sum} to  $j$-eigenspaces for $x$.
\item Equality (ii) arises because $x$ acts the same as $x_{U}$ on $\overline{\freelrb_{U,\ell}}$, by Proposition~\ref{summands-as-submodules-prop}.
\item Isomorphism (iii) arises because
the summands indexed by $U$ with $|U|=\ell$ are permuted transitively by $\symm_n$ with the typical summand for 
$U_0=\{1,2,\ldots,\ell\}$ stabilized by
the subgroup $\symm_{U_0} \cong \symm_\ell$.
Thus this is an induced
$\kk\symm_n$-module, e.g., by applying \cite[Prop. 4.3.2]{Webb}.
\item 
Isomorphism (iv) comes from applying Theorem \ref{chamber-eigenspaces-thm} to $\kk \freelrb_\ell$.
\end{itemize}

The argument for $M = \qfreelrb_n$ is similar. In particular, setting $U_0$ to be the $\F_q$-span of the first $\ell$ standard basis vectors $e_1, e_2, \ldots, e_\ell$ in $(\F_q)^n$, one has
equalities and isomorphisms of $\kk GL_n$-modules

\begin{equation*}
\label{eigenspace-space-direct-sum}
\begin{aligned}
\ker \left( (x^{(q)}-[j]_q)|_{\kk \qfreelrb_{n,\geq \ell}
/\kk \qfreelrb_{n,\geq \ell+1}} \right)
&\overset{(i)}{=} \bigoplus_{\substack{U \subseteq (\F_q)^{n}:\\\dim(U)=\ell}} 
\ker( (x^{(q)}-[j]_q)|_{\overline{\kk \qfreelrb_{U,\ell}}} )
\overset{(ii)}{=} \bigoplus_{\substack{U \subseteq (\F_q)^n :\\\dim(U)=\ell}} 
\ker( (x^{(q)}_{U}-[j]_q)|_{\overline{\kk \qfreelrb_{U,\ell}}} )\\
&\overset{(iii)}{\cong} 
\triv_{GL_{n-\ell}} * \ker\left( (x^{(q)}_{U_{0}}-[j]_q )|_{\kk \qfreelrb_{\ell,\ell}}\right)
\overset{(iv)}{\cong} 
   \begin{cases}
   0 &\text{ if $\ell < j$}\\
   \triv_{GL_{n-\ell}} *  \triv_{GL_j} * \derangementrep^{(q)}_{\ell -j} &\text{ if $\ell \geq j$}
   \end{cases}
\end{aligned}
\end{equation*}
where isomorphisms (i), (ii) and (iv) are
justified exactly as in the $q =1$ proof above. For isomorphism (iii), note (as in the proof of Lemma~\ref{Wachs-lemma}) that $GL_n$ acts transitively on $\ell$-subspaces $U$, and that $U_0$ has $GL_n$-stabilizer subgroup $P_{\ell, n-\ell,}$, so that by \cite[Prop. 4.3.2]{Webb}, 
\[ \bigoplus_{\substack{U \subseteq (\F_q)^n :\\\dim(U)=\ell}} 
\ker( (x^{(q)}_{U}-[j]_q)|_{\overline{\kk \qfreelrb_{U,\ell}}} )\]
carries the $GL_n$-representation induced from the $P_{\ell, n-\ell}$-action on $ \ker\left( (x^{(q)}_{U_{0}}-[j]_q )|_{\kk \qfreelrb_{\ell,\ell}}\right).$
Since every $\AAA \in \kk \qfreelrb_{U_{0},\ell} \cong \kk \qfreelrb_{\ell, \ell}$ is a flag $(A_1, \cdots, A_{\ell})$ with $A_{\ell} = U_0$ it follows that this $P_{\ell, n-\ell}$-action is inflated through the surjection $P_{\ell, n-\ell} \twoheadrightarrow GL_{\ell} \times GL_{n-\ell}$, where the action of $GL_{\ell}$ is as $\ker\left( (x^{(q)}_{U_{0}}-[j]_q )|_{\kk \qfreelrb_{\ell,\ell}}\right)$ and the action of $GL_{n-\ell}$ is trivial. 
\end{proof}

\begin{example}\label{ex:n=2}
We illustrate Theorem~\ref{thm:semigroupeigen} computing the Frobenius map image for each
$j$-eigenspace of $x$ on $\kk\mathscr{F}_n$, or
equivalently the 
$q$-Frobenius map image
for each 
$[j]_q$-eigenspace of $x^{(q)}$ on $\kk\mathscr{F}_n^{(q)}$.
For $n=2,3$, the tables below show these
symmetric functions in their $j^{th}$ row,  decomposed
into columns labeled by $\ell$, indexing each filtration factor from \eqref{vector-space-direct-sum} that contributes a term.

\begin{center}
{\bf Frobenius map images for eigenspaces of $x, x^{(q)}$ on $\kk\mathscr{F}_2, \kk\mathscr{F}_2^{(q)}$:}\\[.1in]
\begin{tabular}{| c||c  c|c c| c c|}\hline
  &$\ell = 0$ & & $\ell = 1$ & & $\ell = 2$ &  \\ \hline\hline
  & $h_2 \cdot \derangement_0$ &$=$ & $h_1 \cdot \derangement_1$ &$=$ & $h_0 \cdot \derangement_2$ & $=$ \\
$j=0$ & $h_2\cdot s_{()}$ & $=$ & $h_1 \cdot 0$ & $=$ & $h_0 \cdot s_{(1,1)}$ & $=$\\
 & $s_{(2)}$ &  & $0$ &  & $s_{(1,1)}$ & \\[5pt] \hline
  & & & $h_{(1,1)}\cdot \derangement_0$ & $=$ & $h_1 \cdot \derangement_1$ & $=$\\
 $j=1$ & & & $h_{(1, 1)}\cdot s_{()}$ & $=$ & $h_1 \cdot 0$ & $=$\\
 & & & $s_{(1,1)} + s_{(2)}$ & & $0$ &\\[5pt] \hline
  & & & & & $h_2 \cdot \derangement_0$ & $=$\\
 $j=2$ & & & & & $h_2 \cdot s_{()}$ & $=$\\
& & & & & $s_{(2)}$ & \\[5pt] \hline
\end{tabular}
\end{center}

\vspace{1cm}

\begin{center}
{\bf Frobenius map images for eigenspaces of $x, x^{(q)}$ on $\kk\mathscr{F}_3, \kk\mathscr{F}_3^{(q)}$:}\\[.1in]
\begin{tabular}{| c||c  c|c c| c c| c c|}\hline 
 & $ \ell = 0$ & & $\ell = 1$ & & $\ell = 2$ & & $\ell = 3$ & \\ \hline \hline
 & $h_{3}\cdot \derangement_0$ & $=$ & $h_{2}\cdot \derangement_{1}$ & $=$ & $h_{1}\cdot \derangement_2$ & $=$ & $h_0 \cdot \derangement_3$ & $=$\\
 $j=0$ & $h_3 \cdot s_{()}$ & $=$ & $h_2 \cdot 0$ & $=$ & $h_1 \cdot s_{(1,1)}$ & $=$ & $h_0 \cdot s_{(2,1)}$ & $=$\\
 & $s_{(3)}$ & & $0$ & & $s_{(2,1)} + s_{(1, 1, 1)}$ & & $s_{(2,1)}$ & \\[5pt] \hline
  & & & $h_{(2,1)}\cdot \derangement_0$ & $=$ & $h_{(1, 1)}\cdot \derangement_1$ & $=$ & $h_{1}\cdot \derangement_2$ & $=$\\
  $j=1$ & & & $h_{(2,1)}\cdot s_{()}$ & $=$ & $h_{(1,1)} \cdot 0$ & $=$ & $h_1 \cdot s_{(1,1)}$ & $=$\\
   & & & $s_{(3)} + s_{(2, 1)}$ & & $0$ & & $s_{(2,1)} + s_{(1, 1, 1)}$ & \\[5pt] \hline
    & & & & & $h_{(2,1)}\cdot \derangement_0$ & $=$ & $h_{2}\cdot \derangement_1$ & $=$\\
  $j=2$ & & & & & $h_{(2,1)}\cdot s_{()}$ & $=$ & $h_2 \cdot 0$ & $=$\\
   & & & & & $s_{(3)} + s_{(2, 1)}$ & & $0$ & \\[5pt] \hline
    & & & & & & & $h_{3}\cdot \derangement_0$ & $=$\\
  $j=3$ & & & & & & & $h_3 \cdot s_{()}$ & $=$\\
   & & & & & & & $s_{(3)}$ & \\[5pt] \hline
\end{tabular}
\end{center}
\end{example}


\bibliographystyle{abbrv}
\bibliography{bibliography}
\end{document}